\documentclass[review]{elsarticle}

\usepackage[margin=1in]{geometry}  
\usepackage{graphicx}              
\usepackage{amsmath}               
\usepackage{amsfonts}              
\usepackage{amsthm}                
\usepackage{algorithm}
\usepackage{algorithmic}
\usepackage{varwidth}
\usepackage{parskip}
\usepackage{hyperref}
\usepackage{rotating}
\usepackage{pdflscape}
\usepackage[numbers]{natbib}
\usepackage{caption}
\usepackage{subcaption}

\newcommand{\bea}{\begin{eqnarray}}
\newcommand{\eea}{\end{eqnarray}}
\newcommand{\bee}{\begin{eqnarray*}}
\newcommand{\eee}{\end{eqnarray*}}
\newcommand{\nn}{\nonumber}
\newcommand{\lb}{\label}
\newcommand{\la}{\lambda}
\newtheorem{thm}{Theorem}[section]
\newtheorem{lem}[thm]{Lemma}

\newtheorem{remark}{Remark}

\begin{document}

\begin{frontmatter}

\title{On the joint distribution of an infinite-buffer discrete-time batch-size-dependent service queue with single and multiple vacation}

\author{N. Nandy}
\ead{nilanjannanday2@gmail.com}
\author{S. Pradhan\corref{mycorrespondingauthor}}
\ead[url]{spiitkgp11@gmail.com}
\cortext[mycorrespondingauthor]{Corresponding author}
\ead{spiitkgp11@gmail.com}
\address{Department of Mathematics, Visvesvaraya National Institute of Technology, Nagpur-440010, India}
\begin{abstract}
Due to the widespread applicability of discrete-time queues in wireless networks or telecommunication systems, this paper analyzes an infinite-buffer batch-service queue with single and multiple vacation where customers/messages arrive according to the Bernoulii process and service time varies with
the batch-size. The foremost focal point of this analysis is to get the complete joint distribution of queue length and server content at
service completion epoch, for which first the bivariate probability generating function has been derived. We also acquire the joint distribution at
arbitrary slot. We also provide several marginal distributions and performance measures for the utilization of the vendor. Transmission of data through a particular channel is skipped due to the high transmission error. As the discrete phase type distribution plays a noteworthy role to control this error, we
include numerical example where service time distribution follows discrete phase type distribution. A comparison between batch-size dependent and independent service has been drawn through the graphical representation of some performance measures and total system cost.
\end{abstract}

\end{frontmatter}
\section{Introduction}
In recent years, Broadband Integrated Services Digital Network (B-ISDN) based on Asynchronous Transfer Mode (ATM) technology, IEEE802.11 n WLANs, circuit switched time-division multiple access (TDMA) systems,  etc. have received abound attention of the researchers mainly due to its compatibility of providing a common interface for multimedia services. The underlying mechanism of those systems is completely based on packet switching principle in which the transmission of packets are allowed only at regularly spaced points in time generally called as slot. Consequently, those systems can be adequately modeled and analyzed using discrete-time queues which have notable diverse spectrum of applications in modern telecommunication/wireless systems, see e.g. Bruneel and Kim \cite{bruneel1993discrete}, Alfa \cite{alfa2010queueing}, Samanta et al. \cite{samanta2007discrete,samanta2007analyzing,samanta2020waiting}, Gupta et al. \cite{gupta2007discrete,gupta2014analysis}, Claeys et al. \cite{claeys2010complete} and references therein.\\
\hspace*{0.3cm}Batch-service queues are ubiquitous in several practical situations such as automatic manufacturing technology (very large-scale integrated (VLSI) circuits), blood pooling, ovens in manufacturing system, mobile crowdsourcing app for smart cities, recreational devices in amusement park etc. A tremendous research is rapidly growing on batch-service queues for the betterment of human civilization, see e.g. Chaudhry and Gupta \cite{chaudhry2003analysis}, Chang et al. \cite{ho2004performance}, Goswami et al. \cite{goswami2006analyzing,goswami2006discrete}, Germs and Foreest \cite{germs2013analysis}, Barbhuiya and Gupta \cite{barbhuiya2019discrete}, Bank and Samanta \cite{bank2020analytical} and references therein. In recent times, a few researchers have focused on batch-service queues with batch-size-dependent service due to their usefulness in production and transportation, package delivery, group testing of blood/urine samples, etc. For more detail see Pradhan et al. \cite{pradhan2015analyzing,pradhan2015queue,pradhan2018IJOR}, Claeys et al. \cite{claeys2010queueing,claeys2011analysis}. \\
\hspace*{0.3cm} For the discrete-time batch-service $Geo^X/G_n^{(l,c)}/1$ and D-BMAP$/G^{(l,c)}_r/1$ queue,
Claeys et al.  \cite{claeys2011analysis,claeys2013analysis}  derived joint probability/vector generating function (pgf/vgf) of queue and server content at
arbitrary slot. Claeys et al. \cite{claeys2013tail} also focused on the analysis of tail probabilities for the customer delay. They also extended their investigation on the influence of correlation of the arrival process on the behavior of the system. However, it should be noted here that the complete extraction procedure of the joint distribution was lacking.  Banerjee et al. \cite{banerjee2014analysis} provided the complete joint distribution of queue and server content for a discrete-time finite-buffer $Geo/G_r^{(a,b)}/1/N$ queue.
Yu and Alfa \cite{yu2015algorithm} analyzed D-MAP$/G_r^{(1,a,b)}/1/N$ queue in which they reported the queue length and server content distribution together using both embedded Markov chain technique (EMCT) and quasi birth and death (QBD) process. These analysis does not focus on any type of vacation.\\
\hspace*{0.3cm}\hspace*{0.3cm}In several real-life circumstances, it is observed that after the service completion of a batch, the server may be unavailable
for a random period of time if the server finds less number of customers than the minimum threshold. The server may be utilize this time in order to carry out some additional work. These type of queues are known as vacation queues which have potential applications in polling protocols frequently used in high-speed telecommunication networks, designing of local area networks, processor schedules in computer and switching systems, shared resources maintenance, manufacturing system with server breakdown etc. The literature clearly exhibits the extensive studies on vacation queues, for example see Doshi \cite{doshi1986queueing}, Lee et al. \cite{lee1994analysis}, Chang et al. \cite{ho2004performance}, Sikdar et al. \cite{sikdar2005analytic,sikdar2005queue}, Gupta et al. \cite{gupta2004finite,gupta2005finite,gupta2006computing}, Reddy et al. \cite{reddy1998analysis,reddy1999non}, Madan et al. \cite{madan2003steady}, Banik et al. \cite{banik2006finite,banik2007gi}, Arumuganathan and Jeyakumar \cite{arumuganathan2004analysis,arumuganathan2005steady}, Ke Jau-Chuan \cite{ke2007batch}, Jeyakumar and Arumuganathan \cite{jeyakumar2011non,jeyakumar2012study} and the references therein. It should be noted here that none of these analysis provides the server content distribution which is an essential tool to increase the serving capability of the server. Moreover no one focuses on batch-size-dependent service policy. In recent times, Gupta et al. \cite{gupta2019finite} have focused on the joint distribution for a finite-buffer batch-size-dependent service queue with single and multiple vacations where vacation time depends on queue length in continuous-time set-up. \\
\hspace*{0.3cm}However to the best of authors' knowledge, the joint queue and server content distribution and performance measures for an infinite-buffer discrete-time batch-size-dependent batch-service queue with single and multiple vacation is not available so far in the literature. From the designing and applicability perspective, one of the major concern is to achieve the most cost-effective combination of performance or reliability of the system
output so that the server utilization and processing rate can be met together. The sever content distribution plays a noteworthy role to calculate the average number of customers/packets with the server by which server's capability can be maximized. With this information one can focus on the optimal control of the concerned queue at pre-implementation stage so that total system cost can be reduced. Moreover, batch-size-dependent service mechanism is a powerful tool to reduce the congestion arising in digital communication systems and to avoid the excessive delays of the transmission of packets.\\
\hspace*{0.3cm}Getting motivated from the above perspective, in this paper, we analyze an infinite-buffer discrete-time queue with the Bernoulli
arrival process, general bulk service ($a,b$) rule, single and multiple vacations and batch-size-dependent service policy where our center of focus is on joint
queue and server content distribution. Practical application of this system and detailed model description are depicted in the subsequent sections. It is worthwhile to mention here that the inclusion of single and multiple vacation together along with batch-size-dependent service makes the mathematical analysis much more complex. Firstly, the steady-state governing equations of the concerned model have been developed. Secondly, we move towards our foremost focal point for the derivation of the bivariate pgf of queue and server content distribution together at service completion epoch. Employing the partial fraction technique, the joint distribution have been successfully extracted and presented in a quite simple and elegant form. The probabilities at arbitrary slot are acquired using the service completion/post transmission epoch probabilities. In order to bring out the essence of this system to the system designer/vendor, several marginal distributions along with pivotal performance measures are also provided. At the endmost, some assorted numerical examples and graphical observations are included to show the usefulness of the analytic procedure which can be fruitful to the readers. In one example, the service time follows the discrete phase type distribution as it significantly monitor the transmission error occurring in telecommunication system. We sketch a comparison between batch-size dependent and independent models through some graphical observations.\\
\hspace*{0.3cm} As a counterpart of discrete-time, in continuous-time set-up, Banerjee et al. \cite{banerjee2012reducing,banerjee2013analysis,banerjee2015analysis}, Pradhan et al. \cite{pradhan2015PEVA,pradhan2019analysis,pradhan2020distribution} have analyzed some finite- and infinite- buffer batch-size-dependent batch-service queues wherein
they provided complete joint distribution of queue and server content at different epochs employing EMCT as well as supplementary variable technique.\\
\hspace*{0.3cm}The rest portion of this paper is organized as follows: next section provides a practical application of the concerned model while Section \ref{MD} describes the model in detail. In Section \ref{GE}, system equations are derived in the steady state and Section \ref{DE} provides the complete procedure for getting joint distribution of queue and server content at service completion epoch. A relationship between the probabilities at service completion and arbitrary epoch is established in Section \ref{AE}. Marginal distributions along with performance measures are sketched in the subsequent section. Several numerical results are appended in Section \ref{NR} followed by the final conclusion.
\section{Practical application of the suggested queueing system}
The above queueing system can be suitably used to model several practical scenarios. One such example is sketched here. The messages, data,
images, videos, signals are first broken into manageable information packets. In ATM multiplexing and switching technology, IEEE802.11 n WLANs, internet
protocol or ethernet etc., those packets are transmitted during a fixed slot as a single entity through a common interface with a minimum threshold and
maximum limit which may be thought as general bulk service ($a,b$) rule. The main advantage is the efficiency and improvement of the quality of service as the construction of only one header per aggregated batch instead of one header per single information unit reduce the time as well as the cost. Whenever the
minimum number of packets are not available for the transmission, the multiplexer enters into the random vacation period with well pre-defined rule and
that time may be used to execute some subsidiary works. In order to prevent the congestion in overload control telecommunication and to bypass the
transmission errors, excessive delays of the time required to transmit the packets, the processing/transmission rate dynamically vary with the size of the packets which is exactly the batch-size-dependent service time policy.

\section{Model description}\lb{MD}
\begin{itemize}

\item \emph{Discrete-time set-up}: Let the time axis be slotted in intervals of equal length where length of each slot is unity. More specifically, let us assume that the time axis be marked by $0,1,2,\ldots,k,\ldots$. Here we discuss the model for late arrival delayed access system (LAS-DA) and therefore a potential arrival takes place in ($k-,k$) and a potential batch departure occurs in ($k,k+$). It may be noted here that under this discrete-time set-up, an arrival and a departure/transmission may take place simultaneously at a slot boundary. Different epochs of LAS-DA system are presented pictorially in figure \ref{figg1}.
\begin{figure}[h!]
\begin{center}
\includegraphics[scale=0.6]{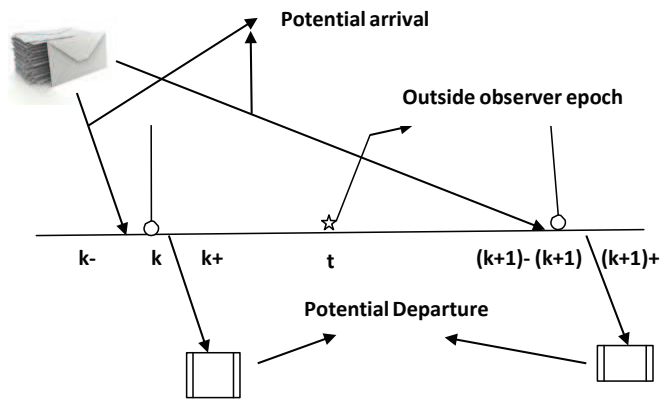}
\end{center}
\caption{Different epochs of LAS-DA system}\lb{figg1}
\end{figure}

\item \emph{Arrival process}: We assume that the customers/messages in terms of packets arrive according to the Bernoulli process with parameter $\la$ as
    it is analytically tractable and does not possess any fractal features like self similarity. Hence the inter-arrival times of the customers/messages are independent and geometrically distributed with probability mass function (pmf) $\psi_n=\bar \la ^{n-1} \la,~0<\la <1,~n\geq 1$
   and $\bar \la=1-\la$.

\item  \emph{Batch-service rule}: The packets are transmitted in group/batches according to the general bulk service ($a,b$) rule. The server only starts the service if the queue contains at least as many customers as the service threshold $`a$'. For the queue size $r$ ($a\leq r\leq b$), entire group of customers are taken for the service. When the queue size exceeds $`b$', then the server can transmit maximum $`b$' packets and others remain in the
    queue for the next round of service. It is worthwhile to mention here that a newly arriving customer/message cannot join the ongoing service even if
    there is a free capacity.

\item  \emph{Service process}: In most of the practical scenarios it is observed that, the transmission/processing times do not follow any specific well
       known probability distribution. Keeping this in mind, here it is assumed that the service times are generally distributed so that a large class of probability distribution can be covered. We also assume that the service time depends on the batch-size under service as this service mechanism plays a pivotal role to reduce congestions and to improve the productivity of the system. Let us define the random variable $T_i$, $a\leq i \leq b$, as the service time of a batch of `$i$' customers with pmf $s_i(n)=\mbox{Pr}(T_i=n),~n=1,2,3,\ldots$ with corresponding pgf $S^{*}_i(z)=\sum_{n=1}^{\infty}s_i(n)z^n$, the mean service time $s_i=\frac{1}{\mu_i}=S_i^{*(1)}(1)$, $a\leq i \leq b$, where $S_i^{*(1)}(1)$ is the first order derivative evaluated at $z=1$.

\item \emph{Vacation process}: The concerned queueing model is analyzed with two type of vacation rules viz., single vacation and multiple vacation using an indicator variable $\delta_p$ as follows:
    \begin{align*}
\delta_p =\left\{\begin{array}{r@{\mskip\thickmuskip}l}
& 0~,\hspace{1.0cm}\mbox{for single vacation}\\
&1,\hspace{1.3cm} \mbox{for multiple vacation}.
\end{array}\right.
\end{align*}
    It has to be kept in mind that the server must decide the pre-defined vacation rule before the service initiation. The results for the corresponding queue with single vacation can be obtained by substituting $\delta_p=0$ and that with multiple vacation by substituting $\delta_p=1$.
    \begin{itemize}
    \item \emph{Single vacation rule}: After the completion of a batch-service, if the server finds at least `$a$' customers waiting, it continues the service process, otherwise it leaves for a vacation of random length with pmf $v_n,~n\geq 1$, pgf $V^*(z)=\sum_{n=1}^{\infty}v_n z^n$  and mean $E(V)=\bar v=V^{*(1)}(1)$. At the vacation termination epoch, if the queue length is still less than the threshold value `$a$', then the server remains dormant till the minimum threshold have been accumulated.

    \item \emph{Multiple vacation rule}: After the completion of a batch-service, if there is at least `$a$' customers waiting in the queue, it continues the service. If the queue length is less than `$a$', then the server leaves for a vacation of random time. After returning from the vacation, if still `$a$' customers have not been accumulated then the server leaves for another vacation and so on until at least `$a$' customers waiting in the queue. The length of random vacation time is generally distributed with pmf  $v_n,~n\geq 1$, pgf $V^*(z)=\sum_{n=1}^{\infty}v_n z^n$  and mean $E(V)=\bar v=V^{*(1)}(1)$.
    \end{itemize}

\item  \emph{Traffic intensity}: The traffic intensity of the system is $\rho=\displaystyle\frac{\la}{b ~\mu_b} <1$, which ensures the stability of the system.
\end{itemize}

\section{Steady-state governing system equations}\lb{GE}
In this section our main goal is to develop the governing equations of the concerned model in the steady-state. Let us first define the states of the system at
time $t$ as:
\begin{itemize}

\item $N_q(k-)$ $\equiv$ Number of packets/customers in the queue waiting to be transmitted,

\item $N_s(k-)$ $\equiv$ Number of packets/customers with the server,

\item $U(k-)$ $\equiv$ Remaining service time of a batch in service (if any),

\item $V(k-)$ $\equiv$ Remaining vacation time of the server excluding the current vacation slot,

\item $\zeta(t-)$ $\equiv$  State of the server defined as:
\begin{align*}
\hspace{1.4cm}\zeta(t-) = \left\{\begin{array}{r@{\mskip\thickmuskip}l}
& 2\hspace{0.8cm} \mbox{when~ the~server~ is ~busy~ },\\
&1\hspace{0.8cm}  \mbox{when ~the~server~ is~ on~ vacation,}\\
&0\hspace{0.8cm}  \mbox{when ~the~server~ is~ in~ the ~dormant~state.}\\
\end{array}\right.
\end{align*}

\end{itemize}
Further, let us define the joint probabilities as:
\bea p_{n,0}(k-)&=&\mbox{Pr}\{ N_q(k-)=n,~ N_s(k-)=0, ~\zeta(t-)=0  \}, \nn\\
p_{n,r}(u,k-)&=&\mbox{Pr}\{ N_q(k-)=n, ~N_s(k-)=r,~ U(k-)=u,~ \zeta(t-)=2 \}, \quad u\geq 1,~n\geq 0,~a\leq r\leq b.\nn\\
Q_{n}(u,k-)&=&\mbox{Pr}\{ N_q(k-)=n, ~V(k-)=u,~ \zeta(t-)=1 \}, \quad u\geq 1,~n\geq 0.\nn \eea
As the model is analyzed in the steady-state, let us define the limiting probabilities as:
\bea p_{n,0}&=&\lim_{k-\rightarrow \infty}p_{n,0}(k-), \nn\\
p_{n,r}(u)&=&\lim_{k-\rightarrow \infty}p_{n,r}(u,k-), ~~ u\geq 1,~n\geq 0,~a\leq r\leq b.\nn \\
Q_{n}(u)&=&\lim_{k-\rightarrow \infty}Q_{n}(u,k-), ~~ u\geq 1,~n\geq 0.\nn\eea
Observing the states of the system at the epochs $k-$ and $(k+1)-$, and using the supplementary variable technique (SVT), we obtain the following equations in the steady-state:
\begin{eqnarray}
p_{0,0} &=& \left(1-{\delta}_{p}\right)\bar{\lambda}p_{0,0} + \left(1-{\delta}_{p}\right)\bar{\lambda}Q_{0}(1)\lb{eq1}\\
p_{n,0} &=&\left(1-{\delta}_{p}\right)\bar{\lambda}p_{n,0} + \left(1-{\delta}_{p}\right)\bar{\lambda}Q_{n}(1) + \left(1-{\delta}_{p}\right)\lambda Q_{n-1}(1), \quad 1\leq n \leq a-1 \lb{eq2}\\
p_{0,a}(u)&=&\bar{\lambda}p_{0,a}(u+1) + \bar{\lambda}\sum_{m=a}^{b}p_{a,m}(1)s_{a}(u) + \lambda \sum_{m=a}^{b}p_{a-1,m}(1)s_{a}(u)+ \bar{\lambda}Q_{a}(1)s_{a}(u)\nn\\
&& ~~~~~~~~~~~~~~~~~~~~~~~~~~~~~~~~~~~~~~~~~~~~~~ + \lambda Q_{a-1}(1)s_{a}(u)+ \lambda\left(1-{\delta}_{p}\right) p_{a-1,0}s_{a}(u)\lb{eq3}\\
p_{0,r}(u)&=&\bar{\lambda}p_{0,r}(u+1)+ \bar{\lambda}\sum_{m=a}^{b}p_{r,m}(1)s_{r}(u) + \lambda \sum_{m=a}^{b}p_{r-1,m}(1)s_{r}(u)+ \bar{\lambda}Q_{r}(1)s_{r}(u)\nn\\
&&~~~~~~~~~~~~~~~~~ ~~~~~~~~~~~~~~~~~~~~~~~~~~~~+ \lambda Q_{r-1}(1)s_{r}(u),\quad a+1 \leq r \leq b\lb{eq4}\\
p_{n,r}(u) &=& \bar{\lambda}p_{n,r}(u+1) + \lambda p_{n-1,r}(u+1), \quad n \geq 1,\quad a \leq r \leq b-1\lb{eq5}\\
p_{n,b}(u) &=& \bar{\lambda}p_{n,b}(u+1)+\lambda p_{n-1,b}(u+1)+\bar{\lambda}\sum_{m=a}^{b}p_{n+b,m}(1)s_{b}(u)+ \lambda \sum_{m=a}^{b}p_{n+b-1,m}(1)s_{b}(u) \nonumber\\
&&~~~~~~~~~~~~~~~~~~~~~~~~~~~~~~~~~~~~~~~~~~~~ + \bar{\lambda}Q_{n+b}(1)s_{b}(u) + \lambda Q_{n+b-1}(1)s_{b}(u),\quad n\geq 1\lb{eq6}\\
Q_{0}(u) &=& \bar{\lambda}Q_{0}(u+1) + \bar{\lambda} \sum_{m=a}^{b}p_{0,m}(1) v(u) + {\delta_{p}}\bar{\lambda} Q_{0}(1)v(u)\lb{eq7}\\
Q_{n}(u) &=& \bar{\lambda}Q_{n}(u+1) + \lambda Q_{n-1}(u+1) + \bar{\lambda} \sum_{m=a}^{b} p_{n,m}(1)v(u) + \lambda \sum_{m=a}^{b} p_{n-1,m}(1)v(u)\nonumber\\
&&~~~~~~~~~~~~~~~~~~~~~~~~~~~~~~~~ +{\delta_{p}} \left[\bar{\lambda}Q_{n}(1)+\lambda Q_{n-1}(1)\right]v(u), \quad 1\leq n \leq a-1 \lb{eq8}\\
Q_{n}(u) &=& \bar{\lambda}Q_{n}(u+1) + \lambda Q_{n-1}(u+1), \quad n \geq a. \lb{eq9}
\end{eqnarray}
Further, let us define
\begin{eqnarray}
p^*_{n,r}(z) &=& \sum_{u=1}^{\infty}p_{n,r}(u) z^{u}, \quad n \geq 0, \quad a \leq r \leq b\nonumber\\
Q^*_{n}(z) &=& \sum_{u=1}^{\infty} Q_{n}(u)z^{u}, \quad n \geq 0.\nonumber
\end{eqnarray}
Consequently, the two results provided below immediately follow which will be used in the analysis.
\begin{eqnarray}
p_{n,r} &\equiv& p^*_{n,r}(1) = \sum_{u=1}^{\infty} p_{n,r}(u), \quad a\leq r \leq b, \quad n\geq 0\nonumber\\
Q_{n} &\equiv& Q^*_{n}(1) = \sum_{u=1}^{\infty} Q_{n}(u), \quad n \geq 0.\nonumber
\end{eqnarray}
Our foremost focal point is to get the joint distributions from (\ref{eq1}) - (\ref{eq9}) for which first we take the $z$-transform of the those equations.
Multiplying (\ref{eq3}) - (\ref{eq9}) by $z^u$ and adding over $u$ from 1 to $\infty$, we obtain
\begin{eqnarray}
\left(\frac{z-\bar{\lambda}}{z}\right)p^*_{0,a}(z) &=&
 \bar{\lambda}\sum_{m=a}^{b}p_{a,m}(1)S^*_{a}(z) + \lambda \sum_{m=a}^{b}p_{a-1,m}(1)S^*_{a}(z) + \bar{\lambda}Q_{a}(1)S^*_{a}(z)\nonumber\\
&& + \lambda Q_{a-1}(1)S^*_{a}(z) +\left(1-\delta_{p}\right) \lambda p_{a-1,0} S^*_{a}(z)-\bar{\lambda}p_{0,a}(1)\lb{eq10}\\
\left(\frac{z-\bar{\lambda}}{z}\right)p^*_{0,r}(z) &=&
 \bar{\lambda}\sum_{m=a}^{b}p_{r,m}(1)S^*_{r}(z) + \lambda \sum_{m=a}^{b}p_{r-1,m}(1)S^*_{r}(z)+ \bar{\lambda}Q_{r}(1)S^*_{r}(z)\nonumber\\
&& + \lambda Q_{r-1}(1)S^*_{r}(z)-\bar{\lambda}p_{0,r}(1), \quad a+1 \leq r \leq b \lb{eq11}\\
\left(\frac{z-\bar{\lambda}}{z}\right)p^*_{n,r}(z) &=&  \frac{\lambda}{z}p^*_{n-1,r}(z)-\bar{\lambda}p_{n,r}(1) - \lambda p_{n-1,r}(1),
\quad n \geq 1, a \leq r \leq b-1 \lb{eq12}\\
\left(\frac{z-\bar{\lambda}}{z}\right)p^*_{n,b}(z) &=&  \frac{\lambda}{z}p^*_{n-1,b}(z) + \bar{\lambda}\sum_{m=a}^{b}p_{n+b,m}(1)S^*_{b}(z)+ \lambda \sum_{m=a}^{b}p_{n+b-1,m}(1)S^*_{b}(z) \nonumber\\
&& + \bar{\lambda}Q_{n+b}(1)S^*_{b}(z)+ \lambda Q_{n+b-1}(1)S^*_{b}(z)-\bar{\lambda}p_{n,b}(1) - \lambda p_{n-1,b}(1) ,\quad n \geq 1 \lb{eq13} \\
\left(\frac{z-\bar{\lambda}}{z}\right)Q^*_{0}(z) &=&  \bar{\lambda} \sum_{m=a}^{b} p_{0,m}(1) V^*(z) + \delta_{p} \bar{\lambda} Q_{0}(1)V^*(z)-\bar{\lambda} Q_{0}(1) \lb{eq14}\\
\left(\frac{z-\bar{\lambda}}{z}\right)Q^*_{n}(z) &=&  \frac{\lambda}{z}Q^*_{n-1}(z)+ \sum_{m=a}^{b}\biggl[\bar{\lambda}p_{n,m}(1) + \lambda p_{n-1,m}(1)\biggr]V^*(z)  \nonumber\\
&&+\delta_{p}\bigg[\bar{\lambda}Q_{n}(1)+\lambda Q_{n-1}(1)\bigg]V^*(z)-\bar{\lambda}Q_{n}(1)-\lambda Q_{n-1}(1), \quad 1 \leq n \leq a-1 \lb{eq15}\\
\left(\frac{z-\bar{\lambda}}{z}\right)Q^*_{n}(z) &=& \frac{\lambda}{z}Q^*_{n-1}(z)-\bar{\lambda}Q_{n}(1)-\lambda Q_{n-1}(1), \quad n \geq a. \lb{eq16}
\end{eqnarray}
Now using equations (\ref{eq1}), (\ref{eq2}) and (\ref{eq10}) - (\ref{eq16}) we derive two results (presented below in the form of lemmas) which will be used in sequel.
\begin{lem}
The probabilities $\left(p_{n,r}^+,p_{n,r}(1)\right)$ and $\left(Q_{n}^+,Q_{n}(1)\right)$ are connected by the relation
\begin{eqnarray}
p^+_{0,r} &=& \tau^{-1}\left\{\lambda p_{0,r}(1)\right\},\quad a \leq r \leq b \lb{eq17}\\
p^+_{n,r} &=& \tau^{-1}\left\{\bar{\lambda}p_{n,r}(1) + \lambda p_{n-1,r}(1)\right\}, \quad n \geq 1, \quad a \leq r \leq b \lb{eq18}\\
Q^+_{0} &=& \tau^{-1}\left\{\bar{\lambda}Q_{0}(1)\right\} \lb{eq19}\\
Q^+_{n} &=& \tau^{-1}\left\{\bar{\lambda}Q_{n}(1) + \lambda Q_{n-1}(1)\right\} ,\quad n \geq 1 \lb{eq20}
\end{eqnarray}
where $\tau = \sum_{m=0}^{\infty}\sum_{r=a}^{b}p_{m,r}(1) + \sum_{m=0}^{\infty} Q_{m}(1).$~\\
Equations (\ref{eq17}) and (\ref{eq18}) provide the relation between the joint queue and server content distribution at service completion epoch with the joint distribution of queue and server content when the service is about to complete. Similarly equations (\ref{eq19}) and (\ref{eq20}) can be interpreted. Here $\tau^{-1}$ gives us the mean of service completion or vacation termination per unit time, i.e., mean departure rate from the busy state or vacation state.\\
\noindent \textbf{Proof}: Now we establish the relationship between $p_{n,r}^+$ and $p_{n,r}(1)$, where $p_{n,r}(1)$ denotes the joint probability that there are $n$ customers in the queue and $r$ with the server and remaining service time is just one slot. Applying Bayes' theorem, we have
\bea p_{n,r}^+&=& \mbox{Pr}\{\mbox{ $n$ customers in the queue and $r$ with the server just prior to the service completion }\nn\\&& \mbox{epoch and a customer arrives $\mid$ total number of customers in the queue just prior to the }\nn\\&& \mbox{service completion epoch of a batch of size  $r$} ~~ \},~n\geq 0,~a\leq r \leq b, \nn\\
&=&  \left\{\begin{array}{r@{\mskip\thickmuskip}l}
& \frac{1}{\tau}\left[\bar \la p_{0,r}(1)\right],\\
& \frac{1}{\tau} \left[\bar \la p_{n,r}(1)+\la p_{n-1,r}(1)\right],~~n\geq 1
\end{array}\right.\nn\eea
where
\bea \tau &=&\mbox{Pr}\{ \mbox{total number of customers in the queue just prior to the service} \nn\\&& \mbox{ completion or vacation termination epoch}\}= \sum_{m=0}^{\infty}\sum_{r=a}^{b}p_{m,r}(1)+\sum_{m=0}^{\infty}Q_{m}(1).\nn\eea
In the similar fashion equations (\ref{eq19}) and (\ref{eq20}) can be established.
\end{lem}

\begin{lem}
The value of $\tau$ is given by
\begin{eqnarray}
\hspace*{-1.0cm}\tau &=& \sum_{m=0}^{\infty}\sum_{r=a}^{b}p_{m,r}(1) + \sum_{m=0}^{\infty} Q_{m}(1)
= \frac{1-\left(1-\delta_{p}\right)\sum_{n=0}^{a-1}p_{n,0}}{\omega}\lb{eq21}\\
\hspace*{-1.0cm}\text{where}\nonumber \\
\hspace*{-1.0cm}\omega &=& \sum_{n=0}^{a-1}\left\{p^+_{n}E(V)+\left(1-\delta_{p}\right)Q^+_{n}s_{a}+\delta_{p} Q^+_{n}E(V)\right\}+\sum_{n=a}^{b}
\left(p^+_{n}+Q^+_{n}\right)s_{n} + \sum_{n=b+1}^{\infty}\left(p^+_{n}+Q^+_{n}\right)s_{b} \lb{eq22}
\end{eqnarray}
\end{lem}

\begin{proof}[\textbf{Proof}]
For single vacation substituting $\delta_{p} = 0 $ in (\ref{eq1}) and (\ref{eq2}), we obtain
\begin{eqnarray}
\lambda p_{n,0} = \bar{\lambda}Q_{0}(1) + \sum_{i=1}^{n}\bigg\{\bar{\lambda}Q_{i}(1) + \lambda Q_{i-1}(1)\bigg\}, \quad 1 \leq n\leq a-1 \lb{eq23}
\end{eqnarray}
Using (\ref{eq23}) in (\ref{eq10}) and then summing over $r$ from $a$ to $b$ and $n$ from 0 to $\infty$ in the equations (\ref{eq10}) - (\ref{eq16}), after some simplification, we get
\begin{eqnarray}
&&\left(\dfrac{z-1}{z}\right)\left\{\sum_{n=0}^{\infty}\sum_{r=a}^{b}p^*_{n,r}(z) + \sum_{n=0}^{\infty}Q^*_{n}(z)\right\}\nn\\
&=& \left[\bar{\lambda}Q_{0}(1) + \sum_{n=1}^{a-1}\biggl\{\bar{\lambda}Q_{n}(1) + \lambda Q_{n-1}(1)\biggr\}\right]\biggl\{\big(1-\delta_{p}\big)S^*_{a}(z) + \delta_{p}V^*(z)\biggr\} \nonumber\\
&&+ \sum_{r=a}^{b}\left[\bar{\lambda}p_{0,r} + \sum_{n=1}^{a-1}\biggl\{\bar{\lambda}p_{n,r}(1) + \lambda p_{n-1,r}(1)\biggr\}\right]V^*(z)\nn\\ &&+\sum_{n=a}^{b}\left[\sum_{r=a}^{b}\biggl\{\bar{\lambda}p_{n,r}(1) + \lambda p_{n-1,r}(1)\biggr\}+\bar{\lambda}Q_{n}(1) + \lambda Q_{n-1}(1)\right]S^*_{n}(z) \nonumber\\
&&+\sum_{n=b+1}^{\infty}\left[\sum_{r=a}^{b}\biggl\{\bar{\lambda}p_{n,r}(1) + \lambda p_{n-1}(1)\biggr\}+\bar{\lambda}Q_{n}(1) + \lambda Q_{n-1}(1)\right]S^*_{b}(z)\nn\\
&& -\sum_{n=0}^{\infty}\sum_{r=a}^{b}p_{n,r}(1) -\sum_{n=0}^{\infty}Q_{n}(1)\nonumber
\end{eqnarray}
Now taking limit $z \rightarrow 1$ in both sides of the above expression, using L'H$\hat{o}$pital's rule and the normalizing condition $\left(1-\delta_{p}\right)\sum_{n=0}^{a-1}p_{n,0} + \sum_{n=0}^{\infty}\sum_{r=a}^{b}p_{n,r} + \sum_{n=0}^{\infty}Q_{n} = 1$, we obtain the desired result.
\end{proof}

\section{\textbf{Joint queue length and server content distribution at service completion epoch of a batch}}\lb{DE}
The primary intention of this section is to provide the complete queue and server content distribution together. Firstly, our center of focus is to derive the bivariate pgf of the joint distribution of queue and server content at the service completion of a batch. Secondly we turn our attention to extract those joint distributions in a simple and elegant way.\\
In view of this, we first define the following pgfs as:
\begin{eqnarray}
P(z, x, y) &=& \sum_{n=0}^{\infty}\sum_{r=a}^{b}p^*_{n,r}(z)x^{n}y^{r}, \quad  |x| \leq 1, \quad  |y| \leq 1 \lb{eq24}\\
P^+(x, y) &=& \sum_{n=0}^{\infty}\sum_{r=a}^{b}p^+_{n,r}x^{n}y^{r}, \quad |x| \leq 1, \quad |y| \leq 1 \lb{eq25}\\
P^+(x, 1) &=& \sum_{n=0}^{\infty}\sum_{r=a}^{b}p^+_{n,r}x^{n}=\sum_{n=0}^{\infty}p^+_{n}x^{n} = P^+(x), \quad |x| \leq 1 \lb{eq26}\\
Q^+(x) &=& \sum_{n=0}^{\infty}Q^+_{n}x^{n}, \quad |x| \leq 1. \lb{eq27}
\end{eqnarray}
Since our major concern is to find the bivariate pgf, multiplying (\ref{eq10}) - (\ref{eq13}) by appropriate powers of $x$ and $y$, summing over $n$ from $0$ to $\infty$ and $r$ from $a$ to $b$, and using (\ref{eq23}), we get

\begin{eqnarray}
&&\biggl\{\dfrac{z-\left(\bar{\lambda}+\lambda x\right)}{z}\biggr\} P(z, x, y)\nn\\
&=& \left(1-\delta_{p}\right)\biggl\{\bar{\lambda}Q_{0}(1) + \sum_{n=1}^{a-1}\bigg(\bar{\lambda}Q_{n}(1) + \lambda Q_{n-1}(1)\bigg)\biggr\}S^*_{a}(z)y^{a}\nonumber\\
&&+\sum_{n=a}^{b}\biggl[\biggl\{\bar{\lambda}Q_{n}(1)+\lambda Q_{n-1}(1)\biggr\} +
 \sum_{r=a}^{b}\biggl\{\bar{\lambda}p_{n,r}(1) + \lambda p_{n-1,r}(1)\biggr\}\biggl]S^*_{n}(z)y^{n} \nonumber\\
&&+ \sum_{n=b+1}^{\infty}\biggl[\biggl\{\bar{\lambda}Q_{n}(1)+\lambda Q_{n-1}(1)\biggr\}+ \sum_{r=a}^{b}\biggl\{\bar{\lambda}p_{n,r}(1) + \lambda p_{n-1,r}(1)\biggr\}\biggl]S^*_{b}(z)x^{n-b}y^{b}\nn\\
&&-\bar{\lambda}\sum_{r=a}^{b}p_{0,r}(1)y^{r} -\sum_{n=1}^{\infty}\sum_{r=a}^{b}\biggl\{\bar{\lambda}p_{n,r}(1) +\lambda p_{n-1,r}(1)\biggr\}x^{n}y^{r} \lb{eq28}
\end{eqnarray}

\noindent Now substituting $z= \bar{\lambda}+\lambda x$ in (\ref{eq28}) and using (\ref{eq17}) - (\ref{eq20}) and (\ref{eq25}), we obtain
\begin{eqnarray}
P^+(x, y) &=&  \left(1-\delta_{p}\right)\sum_{n=0}^{a-1}Q^+_{n}K^{(a)}(x)y^{a} + \sum_{n=a}^{b}p^+_{n} K^{(n)}(x)y^{n} + \sum_{n=a}^{b}Q^+_{n}K^{(n)}(x)y^{n}\nonumber\\
&&+ \sum_{n=b+1}^{\infty}p^+_{n}K^{(b)}(x)x^{n-b}y^{b} + \sum_{n=b+1}^{\infty}Q^+_{n}K^{(b)}(x)x^{n-b}y^{b} \lb{eq29}
\end{eqnarray}
where $K^{(r)}(x)=S^*_{r}(\bar{\lambda}+\lambda x)$ is the pgf of $k_j^{(r)},~(a\leq r \leq b,~j\geq 0)$ and
\bea k_j^{(r)}=\mbox{Pr}\{j ~\mbox{arrivals during the service time of a batch of size}~ r\}.\nn\eea
Now multiplying (\ref{eq14}) - (\ref{eq16}) with appropriate powers of $x$ and summing over $n$ from $0$ to $\infty$, we get
\begin{eqnarray}
\biggl\{\dfrac{z-\left(\bar{\lambda}+\lambda x\right)}{z}\biggr\}
\sum_{n=0}^{\infty}Q^*_{n}(z)x^{n} &=& \sum_{m=a}^{b}\biggl\{p_{0,m}(1)+\sum_{n=1}^{a-1}\biggl(\bar{\lambda}p_{n,m}(1)+\lambda p_{n-1,m}(1)\biggr)\biggr\}x^{n}V^{*}(z)\nonumber\\
&&+\delta_{p}\biggl\{\bar{\lambda}Q_{0}(1) + \sum_{n=1}^{a-1}\bigg(\bar{\lambda}Q_{n}(1) + \lambda Q_{n-1}(1)\bigg)\biggr\}x^{n}V^*(z)\nonumber\\
&&-\bar{\lambda}Q_{0}(1)-\sum_{n=1}^{\infty}\biggl\{\bar{\lambda}Q_{n}(1)+\lambda Q_{n-1}(1)\biggr\}x^{n} \lb{eq30}
\end{eqnarray}

\noindent Putting $z=\bar{\lambda}+\lambda x$ in (\ref{eq30}) and using (\ref{eq17}) - (\ref{eq20}) and (\ref{eq27}), we get
\begin{eqnarray}
Q^{+}(x) = \sum_{n=0}^{a-1}\left(p^+_{n} + \delta_{p} Q^+_{n}\right)x^{n}H(x) \lb{eq31}
\end{eqnarray}
where $H(x)=\sum_{j=0}^{\infty} h_j x^j=V^*(\bar{\lambda}+\lambda x)$ is the pgf of $h_j$ and
\bea h_j=\mbox{Pr}\{j ~\mbox{arrivals during the vacation time of the server}\}.\nn\eea
Now substituting $y=1$ in (\ref{eq29}) and using (\ref{eq26}) and (\ref{eq31}) and after some algebraic manipulation, we obtain
\begin{eqnarray}
P^{+}(x) =\frac{\begin{aligned}\sum_{n=0}^{a-1}\left[p^+_{n}x^{n}\biggl\{H(x)-1\biggr\}K^{(b)}(x) +Q^+_{n}\left\{\left(1-\delta_{p}\right)x^{b}K^{(a)}(x)\right.\right.\\ \left.\left.+\big(\delta_{p}H(x)-1\big)x^{n}K^{(b)}(x)\right\}\right] +\sum_{n=a}^{b-1}\biggl[\biggl(p^+_{n}+Q^+_{n}\biggr)\biggl\{x^{b}K^{(n)}(x)-x^{n}K^{(b)} (x)\biggr\}\biggr] \end{aligned}}{x^{b}-K^{(b)}(x)} \lb{eq32}
\end{eqnarray}
The above expression represents the pgf of only queue length distribution at service completion epoch.\\
Now looking back into (\ref{eq29}), using (\ref{eq31}) and (\ref{eq32}), after some algebraic simplification, we get
\begin{eqnarray}
P^{+}(x,y) = \frac{\begin{aligned}\sum_{n=0}^{a-1}\biggl[p^+_{n}x^{n}\biggl\{H(x)-1\biggr\}y^{b}K^{(b)}(x)\\
+Q^+_{n}\biggl\{\left(1-\delta_{p}\right)\bigg(\left(x^{b}-K^{(b)}(x)\right)y^{a}+y^{b}K^{(b)}(x)\bigg)K^{(a)}(x)\\
+\left(\delta_{p}H(x)-1\right)x^{n}y^{b}K^{(b)}(x)\biggr\}\biggr] \\+\sum_{n=a}^{b-1}\biggl(p^+_{n}+Q^+_{n}\biggr) \biggl\{\left(x^{b}-K^{(b)}(x)\right)\biggl(y^{n}K^{(n)}(x)-x^{n-b}y^{b}K^{(b)}(x)\biggr) \\
+\biggl(K^{(n)}(x)-x^{n-b}K^{(b)}(x)\biggr)y^{b}K^{(b)}(x)\biggr\}\end{aligned}}{x^{b}-K^{(b)}(x)} \lb{eq33}
\end{eqnarray}
The above expression designates the bivariate pgf of queue and server content distribution at service completion epoch, which is the key ingredient of our analysis. No such result is available in the literature so far to the best of authors' knowledge. One can grab the joint distribution by inverting this bivariate pgf which is discussed in the subsequent analysis. It may be remarked here that direct use of transition probability matrix would lead the same expression. However, we feel that use of supplementary variable technique reduces the complexity involved in the derivation for the same.~\\
\begin{remark}
The expression presented in (\ref{eq32}) is the pgf of \emph{only queue length distribution} at service completion epoch, which plays a noteworthy role to
calculate the tail distribution instead of using the bivariate pgf. In order to compute the cell loss ratio in ATM switching network, the tail distribution is
a powerful tool. To the best of authors' knowledge, this result is also not available so far in the literature. As a consequence, we hope that this also leads to a new contribution in the queueing literature.
\end{remark}

\subsection{\textbf{Steady-state conditions and determination of unknowns in the numerator of (\ref{eq33})}}
Although we have assumed the stability of the concerned queue in the model description, it can be proved using the stability condition given in Abonikov and Dukhovny \cite{abolnikov1991markov}. We conclude that the corresponding Markov chain is ergodic if and only if $\frac{d}{dx}K^{(b)}(x) < b$. Consequently, the steady-state distribution exists if $\frac{\la}{\mu_b}<b$, i.e., $\displaystyle \frac{\la}{b~\mu_b}=\rho < 1$.\\
\hspace*{0.3cm}Before extracting the joint probabilities from the bivariate pgf presented in (\ref{eq33}), we need to find the unknown quantities $p_n^+,~(0\leq n \leq b-1)$ and $Q^+_{n},~(0\leq n \leq a-1)$, appearing in the numerator. As both the expressions presented in (\ref{eq33}) and (\ref{eq32}) contains the same unknown quantities, we use (\ref{eq32}) instead of (\ref{eq33}) to find those unknowns, without loss of any generality. Although the procedure of determining of these unknowns is standard in the literature, for the sake of completeness we describe very briefly. Here we need to determine total $a+b$ unknowns. Using the result given below in the lemma form, we express $Q^+_{n}$ in terms of $p_n^+$ which eventually leads to total $b$ unknowns. Employing Rouch$\acute{\mbox{e}}$'s theorem, it is known that, $x^b-K^{(b)}(x)=D(x)$, (say) has exactly $b$ zeroes in the unit disk $|x|\leq 1$. Using these zeroes and normalizing condition, we obtain total $b$ linear simultaneous equations and hence solving them, we get those unknowns.
\begin{lem}The relation between $p^+_{n}$ and $Q^+_{n}$ are given by:
\begin{eqnarray}
Q^+_{n} &=& \sum_{i=0}^{n}\xi_i~p_{n-i}^+ ,\quad n=0,1,2,\ldots,a-1, ~~~\mbox{where}\nn\\
\xi_0 &=& \frac{h_0}{1-h_0}\nn\\
\xi_n &=& \frac{1}{1-h_0}\left[h_n+\sum_{i=1}^{n}h_i~\xi_{n-i}\right].\nn
\end{eqnarray}
\end{lem}

\subsection{\textbf{Extraction of probabilities from the known bivariate pgf}}
The completely known bivariate pgf, the key ingredient of our analysis, is in our grip. In this section, we now proceed to extract the joint probabilities using partial fraction technique. First we accumulate the coefficients of $y^j,~a\leq j \leq b$ from the bivariate pgf (\ref{eq33}) and are given below. \\
Coefficient of $y^{a}$ :
\begin{eqnarray}
\sum_{n=0}^{\infty}p^+_{n,a}x^{n} &=& \left(1-\delta_{p}\right)K^{(a)}(x)
\sum_{n=0}^{a-1}Q^+_{n} + \left(p^+_{a}+Q^+_{a}\right)K^{(a)}(x) \lb{eq34}
\end{eqnarray}
Coefficient of $y^j, a+1 \leq j \leq b-1$ :
\begin{eqnarray}
\sum_{n=0}^{\infty}p^+_{n,j}x^{n} &=& \left(p^+_{j}+Q^+_{j}\right)K^{(j)}(x) \lb{eq35}
\end{eqnarray}
Coefficient of $y^b$ :
\begin{eqnarray}
\sum_{n=0}^{\infty}p^+_{n,b}x^{n} &=& \dfrac{1}{x^{b}-K^{(b)}(x)}\left[\sum_{n=0}^{a-1}p^+_{n}x^{n}\bigg(H(x)-1\bigg)K^{(b)}(x)+
\left(1-\delta_{p}\right)K^{(a)}(x)K^{(b)}(x)\sum_{n=0}^{a-1}Q^+_{n}\right.\nonumber\\
&& \left.+\left(\delta_{p} H(x)-1\right)\sum_{n=0}^{a-1}Q^+_{n}x^{n}K^{(n)}(x)
+\sum_{n=a}^{b-1}\left(p^+_{n}+Q^+_{n}\right)\left(K^{(n)}(x)-x^{n-b}K^{(b)}(x)\right)K^{(b)}(x)\right]\nonumber\\
&&-\sum_{n=a}^{b-1}\left(p^+_{n}+Q^+_{n}\right)x^{n-b}K^{(b)}(x) \lb{eq36}
\end{eqnarray}
Collecting the coefficients of $x^n$ from both the side of (\ref{eq34}) and (\ref{eq35}), we obtain the joint distribution as:
\bea p^+_{n,a}&=&\left[\left(1-\delta_{p}\right)
\sum_{n=0}^{a-1}Q^+_{n} + \left(p^+_{a}+Q^+_{a}\right)\right]k_n^{(a)}\lb{eq37}\\
 p^+_{n,j}&=&\left(p^+_{j}+Q^+_{j}\right)k_n^{(j)},\quad a+1 \leq j \leq b-1. \lb{eq38}\eea
\hspace*{0.3cm}Now we are left with only $p_{n,b}^+$ for which we need to invert the right hand side of (\ref{eq36}). For the known service and vacation
time distributions the right hand side of (\ref{eq36}) is completely known functions of $x$. The corresponding inversion process is discussed in the
subsequent analysis.\\
\hspace*{0.3cm} After substituting $K^{(r)}(x)=S^*_r(\bar \la + \la x),~a\leq r \leq b,$ and $H(x)=V^*(\bar \la + \la x)$ in the expression (\ref{eq36}),
 let us denote numerator of (\ref{eq36}) as $\Lambda(x)$ and denominator as $D(x)$. The degrees of $\Lambda(x)$ and $D(x)$ depend on the distributions of service/vacation time and service threshold value `$a$' and maximum capacity `$b$'. Let the degree of numerator be $L_1$ and that of denominator be $M_1$, respectively. Now, we obtain $p_{n,b}^+$ in terms of roots of $D(x)=0$ by discussing the following cases:

\begin{itemize}
\item \textbf{When all the zeroes of $D(x)$ in $|x|>1$ are distinct}\\
Let us assume $\alpha_{1},\alpha_{2},\ldots,\alpha_{M_1}$ to be the roots of $D(x)=0$, out of which $M_1-b$ are the distinct roots in $|x| > 1$, as the degree of $D(x)$ is $M_1$ and $b$ insides roots are cancelled out with the roots of the numerator. It is also assumed that $D(x)=0$ has $b$ simple roots inside the unit circle. \\
~~\\
\emph{Case 1}: $L_1\geq M_1$\\
\hspace*{0.5cm}Applying the partial-fraction expansion, the rational function $\sum_{n=0}^{\infty}p_{n,b}^+x^n$ can be uniquely written as
\bea \sum_{n=0}^{\infty}p_{n,b}^+x^n=\sum_{i=0}^{L_1-M_1}\tau_{i} x^i +\sum_{k=1}^{M_1-b}\frac{c_{k}}{\alpha_{k}-x},\lb{eq39}\eea
for some constants $\tau_{i}$ and $c_{k}$'s. The first sum is the result of the division of the polynomial $\Lambda(x)$ by $D(x)$ and the constants $\tau_{i}$ are the coefficients of the resulting quotient. Using the residue theorem, we have
$ c_{k}=-\frac{\Lambda(\alpha_{k})}{D^{\prime}(\alpha_{k})},\quad k=1,2,\ldots,M_1-b. $\\
Now, collecting the coefficient of $x^n$ from both the sides of (\ref{eq39}), we have
\bea p_{n,b}^+= \left\{\begin{array}{r@{\mskip\thickmuskip}l}
&\tau_{n}+\sum\limits_{k=1}^{M_1-b}\frac{c_{k}}{\alpha^{n+1}_{k}},\quad 0\leq n\leq L_1-M_1,\\
&\\
& \sum\limits_{k=1}^{M_1-b}\frac{c_{k}}{\alpha^{n+1}_{k}},\quad n> L_1-M_1~.\end{array}\right.\lb{eq40}\eea
\emph{Case 2}: $L_1< M_1$\\
Using partial-fraction technique on $\sum_{n=0}^{\infty}p_{n,b}^+ x^n$, we have
\bea \sum_{n=0}^{\infty}p_{n,b}^+x^n=\sum_{k=1}^{M_1-b}\frac{c_{k}}{\alpha_{k}-x},\label{eq41}\eea
where $c_{k}=-\frac{\Lambda(\alpha_{k})}{D^{\prime}(\alpha_{k})},\quad k=1,2,\ldots,M_1-b. $\\
Now, collecting the coefficient of $x^n$ from both the sides of (\ref{eq41}), we obtain
\bea p_{n,b}^+ &=& \sum\limits_{k=1}^{M_1-b}\frac{c_{k}}{\alpha^{n+1}_{k}},\quad n\geq 0. \lb{eq42}\eea

\item \textbf{When some zeroes of $D(x)$ in $|x|>1$ are repeated}\\
The denominator $D(x)$ may have some multiple zeroes with absolute value greater than one. Let $D(x)$ has total $m$ multiple zeroes, say $\beta_{1},\beta_{2},\ldots,\beta_{m}$ with multiplicity $\pi_{1},\pi_{2},\ldots,\pi_{m}$, respectively. Further it is clear that $D(x)$ has total $(M_1-b-\chi)$ distinct zeroes, where $ \chi=\sum_{i=1}^{m}\pi_{i}$, say, $\alpha_{1},\alpha_{2},\ldots,\alpha_{M_1-b-\chi}$. It is also assumed that $D(x)$ has $b$ simple zeroes inside the unit circle.\\
~~\\
\emph{Case 1}: $L_1\geq M_1$\\
Applying the partial-fraction method, we can uniquely write $\sum_{n=0}^{\infty}p_{n,b}^+x^n,$ as
 \bea \hspace*{-1.0cm}\sum_{n=0}^{\infty}p_{n,b}^+x^n=\sum_{i=0}^{L_1-M_1}\tau_{i} x^i+\sum_{k=1}^{M_1-b-\chi}\frac{c_{k} }{\alpha_{k}-x}+\sum_{\nu=1}^{m}\sum_{i=1}^{\pi_{\nu}}\frac{\zeta_{\nu,i}}{(\beta_{\nu}-x)^{\pi_{\nu}-i+1}} \label{eq43}. \eea
Using residue theorem, we obtain
  \begin{small}
\bea \hspace*{-1.2cm}c_{k}&=&-\frac{\Lambda(\alpha_{k})}{D^{\prime}(\alpha_{k})},\quad k=1,2,\ldots,M_1-b-m, \nn \\
 \hspace*{-1.2cm}\zeta_{\nu,i}&=&
 \frac{1}{(\pi_{\nu}-i)!}~\lim_{x \rightarrow \beta_{\nu}}~\frac{d^{(\pi_{\nu}-i)}}{dx^{(\pi_{\nu}-i)}}\left[ \frac{(\beta_{\nu}-x)^{\pi_{\nu}}\Lambda(x)}{D(x)}\right],~ \nu=1,2,\ldots,m, ~ i=1,2,\ldots,\pi_{\nu}.\nn \eea
 \end{small}
Now collecting the coefficient of $x^n$ from both the sides of (\ref{eq43}), we have the distribution $p_{n,b}^+~,~n\geq 0$ as:
\begin{small}
\bea \hspace*{-1cm}p_{n,b}^+=\displaystyle \left\{\begin{array}{r@{\mskip\thickmuskip}l}
&\tau_{n}+\sum\limits_{k=1}^{M_1-b-\chi}\frac{c_{k}}{\alpha^{n+1}_{k}}+\displaystyle
\sum_{\nu=1}^{m}\sum_{i=1}^{\pi_{\nu}}
\binom{\pi_{\nu}+n-i}{\pi_{\nu}-i}\frac{\zeta_{\nu, i}}{\beta_{\nu}^{\pi_{\nu}+n+1-i}},~~ 0\leq n\leq L_1-M_1,\\
& \\
& \sum\limits_{k=1}^{M_1-b-\chi}\frac{c_{k}}{\alpha^{n+1}_{k}}+\displaystyle
\sum_{\nu=1}^{m}\sum_{i=1}^{\pi_{\nu}}
\binom{\pi_{\nu}+n-i}{\pi_{\nu}-i}\frac{\zeta_{\nu, i}}{\beta_{\nu}^{\pi_{\nu}+n+1-i}},~~ n> L_1-M_1.\end{array}\right. \lb{eq44}\eea
\end{small}
~~\\
\emph{Case 2}: $L_1< M_1$\\
 Here, in partial-fraction, we omit only  the first summation term of the right hand side of (\ref{eq43}). Then collecting the coefficients of $x^n$, we obtain $p_{n,b}^+$, which are given by
 \begin{small}
\bea \hspace*{-1cm}p_{n,b}^+=\sum\limits_{k=1}^{M_1-b-\chi}\frac{c_{k}}{\alpha^{n+1}_{k}}+\sum_{\nu=1}^{m}\sum_{i=1}^{\pi_{\nu}}
\binom{\pi_{\nu}+n-i}{\pi_{\nu}-i}\frac{\zeta_{\nu, i}}{\beta_{\nu}^{\pi_{\nu}+n+1-i}},~ n\geq 0.\lb{eq}\eea
\end{small}
\end{itemize}
This completes the extraction of the joint probabilities $p_{n,b}^+$.
\begin{remark}
One may feel the necessity of having the probabilities at vacation termination epoch i.e., $Q_n^+,~n\geq 0$, which can be easily extracted from $Q^+(x)$ and are given below.
\bea Q_n^+ = \displaystyle \left\{\begin{array}{r@{\mskip\thickmuskip}l}
&\sum_{i=0}^{n}\left(p_i^++\delta_{p}Q_i^+\right)h_{n-i}, \quad 0\leq n \leq a-1 \\
 & \sum_{i=0}^{a-1}\left(p_i^++\delta_{p}Q_i^+\right)h_{n-i},\quad n \geq a. \end{array}\right. \lb{eq46}\eea
\end{remark}
\begin{remark}
It should also be noted here that the queue length distribution at service completion epoch i.e., $p_n^+$ can be achieved by summing over $r$
 from $a$ to $b$ the joint probabilities $p^+_{n,r}$, without the requirement of further separate extraction from $P^+(x)$.
\end{remark}
This completes the determination of joint distribution of queue length and server content at service completion epoch as well as
the only queue length distribution at vacation termination epoch. Now in the next section, the probability distribution at arbitrary
slot have been acquired.
\section{Joint distribution of queue and server content at arbitrary slot}\lb{AE}
The steady-state queue and server content distribution at arbitrary slot plays an important role in computing several notable key performance measures.
A relationship between the probabilities at service completion/post transmission and arbitrary epochs has been generated and presented below.\\
\begin{thm}
The probabilities at service completion and arbitrary epochs i.e., $\left(p_{n,r}, p^+_{n,r}, p^+_{n}\right)$ $\left(Q_{n},Q^+_{n}\right)$ are related by
\begin{eqnarray}
p_{n,0} &=& \dfrac{1}{E^*}\sum_{m=0}^{n}Q^+_{m},\quad 0 \leq n \leq a-1 \lb{eq47}\\
p_{0,a} &=& \dfrac{1}{E^{*}}\biggl(p^+_{a}+Q^+_{a}-p^+_{0,a}+
\left(1-\delta_{p}\right)\sum_{m=0}^{a-1}Q^+_{m}\biggr) \lb{eq48}\\
p_{0,r} &=& \dfrac{1}{E^{*}}\biggl(p^+_{r}+ Q^+_{r} - p^+_{0,r}\biggr), \quad a+1 \leq r \leq b \lb{eq49}\end{eqnarray}
\begin{eqnarray}
p_{n,r} &=& p_{n-1,r} - \dfrac{p^+_{n,r}}{E^*},\quad n \geq 1, \quad a \leq r \leq b-1 \lb{eq50}\\
p_{n,b} &=& p_{n-1,b} + \dfrac{1}{E^*}\biggl(p^+_{n+b}+Q^+_{n+b}-p^+_{n,b}\biggr), \quad n \geq 1 \lb{eq51} \\
Q_{0} &=& \dfrac{1}{E^*}\biggl(p^+_{0}-\left(1-\delta_{p}\right)Q^+_{0}\biggr) \lb{eq52}\\
Q_{n} &=& Q_{n-1} + \dfrac{1}{E^*}\biggl(p^+_{n}-\left(1-\delta_{p}\right)Q^+_{n}\biggr), \quad 1 \leq n \leq a-1 \lb{eq53}\\
Q_{n} &=& Q_{n-1} - \dfrac{Q^+_{n}}{E^*}, \quad n \geq a \lb{eq54}
\end{eqnarray}
where
\begin{eqnarray}
E^* &=& \lambda \omega + \left(1-\delta_{p}\right)\sum_{i=0}^{a-1}\sum_{m=0}^{i} Q^+_{m},\nonumber
\\
\omega &=& \sum_{n=0}^{a-1}\biggl\{p^+_{n}E(V)+\left(1-\delta_{p}\right)Q^+_{n}s_{a}+\delta_{p} Q^+_{n}E(V)\biggr\}+\sum_{n=a}^{b}
\biggl(p^+_{n}+Q^+_{n}\biggr)s_{n} + \sum_{n=b+1}^{\infty}\biggl(p^+_{n}+Q^+_{n}\biggr)s_{b}\nonumber
\end{eqnarray}
\end{thm}
\begin{proof}[\textbf{Proof}]
Use of (\ref{eq23}) in (\ref{eq1}) - (\ref{eq2}), and then division by $\tau$ and use of (\ref{eq17}) - (\ref{eq20}) leads
\begin{eqnarray}
p_{0,0} &=& \dfrac{1-\left(1-\delta_{p}\right)\sum_{i=0}^{a-1}p_{i,0}}{\lambda \omega} Q^+_{0} \lb{eq55}\\
p_{n,0} &=& \dfrac{1-\left(1-\delta_{p}\right)\sum_{i=0}^{a-1}p_{i,0}}{\lambda \omega} \sum_{m=0}^{n} Q^+_{m},\quad 1\leq n \leq a-1. \lb{eq56}
\end{eqnarray}
Now dividing (\ref{eq56}) by (\ref{eq55}), we get
\begin{eqnarray}
p_{n,0} &=& \dfrac{p_{0,0}}{Q^+_{0}} \sum_{m=0}^{n} Q^+_{m}, \quad 0 \leq n \leq a-1. \lb{eq57}
\end{eqnarray}
Use of (\ref{eq57}) in (\ref{eq55}) and (\ref{eq56}) yields
\begin{eqnarray}
p_{n,0} &=& \dfrac{\sum_{m=0}^{n}Q^+_{m}}{\lambda \omega + \left(1-\delta_{p}\right)\sum_{i=0}^{a-1}\sum_{m=0}^{i} Q^+_{m}}, \quad 0 \leq n \leq a-1. \lb{eq58}
\end{eqnarray}
which is the desired result of (\ref{eq47}).\\
Putting $z=1$ in (\ref{eq10}) and (\ref{eq11}), use of (\ref{eq23}),  and division by $\tau$ gives
\begin{eqnarray}
p_{0,a} &=& \dfrac{1-\left(1-\delta_{p}\right)\sum_{i=0}^{a-1}p_{i,0}}{\lambda \omega} \biggl(p^+_{a}+Q^+_{a}+ \left(1-\delta_{p}\right)\sum_{m=0}^{a-1}Q^+_{m}-p^+_{0,a}\biggr) \lb{eq59}\\
p_{0,r} &=& \dfrac{1-\left(1-\delta_{p}\right)\sum_{i=0}^{a-1}p_{i,0}}{\lambda \omega}\biggl(p^+_{r}+Q^+_{r}-p^+_{0,r}\biggr),
\quad a+1 \leq r \leq b  \lb{eq60}
\end{eqnarray}
As
\begin{eqnarray}
1- \left(1-\delta_{p}\right)\sum_{i=0}^{a-1}p_{i,0} &=& \dfrac{\lambda \omega}{\lambda \omega + (1-\delta_{p})\sum_{i=0}^{a-1}\sum_{m=0}^{i}Q^+_{m}}, \lb{eq61}
\end{eqnarray}
Using (\ref{eq61}) in (\ref{eq59}) and (\ref{eq60}), we get the required results of (\ref{eq48}) and (\ref{eq49}).\\
Substituting $z=1$ in (\ref{eq12}) - (\ref{eq16}) and then dividing by $\tau$, using the similar approach, we get our desired results of (\ref{eq50}) - (\ref{eq54}).
\end{proof}
\section{Marginal distributions and performance measures}\lb{MP}
From the application perspective of the concerned model, the marginal distributions along with the performance measures are the key ingredients for the system designers/vendors. We first present the following utmost marginal distributions in terms of known distributions :
\begin{itemize}
\item  the distribution of the number of customers in the system (which includes total the number of customers in the queue and with the server), $p_n^{sys},~~n\geq 0$,
\begin{align*}
p_n^{sys} = \left\{\begin{array}{r@{\mskip\thickmuskip}l}
&\left(1-\delta_{p}\right)p_{n,0} +Q_n ~~,\hspace{2.0cm}0\leq n \leq a-1\\
& \sum_{r=a}^{min(b,n)}p_{n-r,r}+Q_n~~,\hspace{1.3cm} a\leq n \leq b,\vspace{0.1cm}\\
&\sum_{r=a}^{b}p_{n-r,r}\hspace{3.3cm} n\geq b+1.
\end{array}\right.
\end{align*}

\item the distribution of the number of customers in the queue $p_n^{queue},~~n\geq 0$,
\begin{align*}
p_n^{queue} =\left\{\begin{array}{r@{\mskip\thickmuskip}l}
& \left(1-\delta_{p}\right)p_{n,0}+\sum_{r=a}^{b}p_{n,r}+Q_n~,\hspace{1.3cm}0\leq n \leq a-1\\
&\sum_{r=a}^{b}p_{n,r}+Q_n\hspace{4.0cm} n\geq a.
\end{array}\right.
\end{align*}

\item the probability that the server is in busy state ($P_{busy}$), in the vacation state ($Q_{vac}$), and in the dormant state ($P_{dor}$) are given by $P_{busy}=\sum_{n=0}^{\infty}\sum_{r=a}^{b}p_{n,r}$,~ $Q_{vac}=\sum_{n=0}^{\infty}Q_n$,~ $P_{dor}=\sum_{n=0}^{a-1}p_{n,0}$.

\item the conditional distribution of the number of customers undergoing service with the server given that the server is busy,
$p_r^{ser}$, ($a \leq r \leq b $)
\begin{align*}
p_r^{ser} = \sum_{n=0}^{\infty}p_{n,r}/P_{busy} ~~(a\leq r \leq b).
\end{align*}

\end{itemize}
\hspace*{0.3cm}Performance measures are very crucial from the application point of view as it improves the efficiency of the system. Although some performance measures can be derived from the pgf, here we obtain several pivotal performance measures using completely known state probabilities and marginal distributions
in a very simplified manner and are given by:
\begin{itemize}
\item average number of customers waiting in the queue at any arbitrary time ($L_q)=\sum_{n=0}^{\infty}n p_n^{queue}$,
\item average number in the system ($L)=\sum_{n=0}^{\infty}n p_n^{sys}$,
\item average number of customers with the server ($L_s)=\sum_{r=a}^{b}r p_r^{ser}$,
\item average waiting time of a customer in the queue $(W_q)=\frac{L_q}{\la}$ as well as in the system $(W)=\frac{L}{\la}$.
\end{itemize}
\section{Numerical illustration}\lb{NR}
The principal objective of this section is to manifest the feasibility and applicability of the proposed method and results by the inclusion of numerical examples. Although we have generated extensive computations works, only a few of them are appended here due to lack of space. All the calculations were performed using Maple 15 on PC having configuration Intel (R) Core (TM) i5-3470 CPU Processor @ 3.20 GHz with 4.00 GB of RAM, and the results are presented in the tabular form in which some relevant performance measures are also included.\\
\textbf{Example 1:} Transmission errors are the inescapable part of the communication channel as it may occur due to electrical faults, bad weather, fading channel etc. In order to monitor this error, discrete phase type distribution (PH$_D$) significantly fits as a powerful tool. It covers a wide range of almost
all relevant discrete distributions such as geometric distribution, negative binomial distribution etc. On account of this, in this example we consider service time to follow PH$_D$ distribution with the representation ($\boldsymbol{\beta}_i,~\textbf{T}_i$) for $a\leq i \leq b$, where $\boldsymbol{\beta}_i$ is a row vector of order $\nu$ and $\textbf{T}_i$ is a square matrix of order $\nu$. The associated  pgf is $z \boldsymbol{\beta}_i\left(I-z\textbf{T}_i\right)^{-1}\eta_i$ with $\eta_i+\textbf{T}_ie=e$, where $I$ is the identity matrix of appropriate order and $e$ is the column vector. Consequently,
$K^{(i)}(x)=(1-\la+\la x) \boldsymbol{\beta}_i\left(I-(1-\la+\la x)\textbf{T}_i\right)^{-1}\eta_i$.  Here we consider $a=5,~b=10,~\la=0.5,~\rho=0.693181$, and PH$_D$ distribution for different batch-size is provided in the table \ref{tb1}. The vacation time follows geometric distribution with pmf $(v_j)=\eta~(1-\eta)^{j-1},~~j\geq 1$ with $\eta=0.3$, pgf $=\dfrac{\eta z}{1-(1-\eta)z}$, mean ($\bar v$)=3.333333 and hence $H(x)=\dfrac{0.15+0.15x}{0.65-0.35x}$. The joint distribution at service completion and arbitrary epochs along with some performance measures are displayed in Tables [\ref {tb2}-\ref {tb5}] .
\begin{table}[h!]
$\vspace{0.1cm}$
\centering
\caption{PH$_D$ representation for service time of batch size $r$}
\lb{tb1}
\begin{tiny}
\begin{tabular}{|c|c|c|c|}\hline
$r$  &  $\boldsymbol{\beta}_r$  &  $\textbf{T}_r$   &  $E(T_r)=s_r$\\\hline
&&&\\
     5&    $\left(
     \begin{array}{ccc}
     0.3 & 0.4 & 0.3  \\
     \end{array}
     \right)$         &  $\left(
                                                        \begin{array}{ccc}
                                                          0.6 & 0.2 & 0.2  \\
                                                          0.2 & 0.6 & 0.1 \\
                                                          0.2 & 0 & 0.7  \\
                                                        \end{array}
                                                      \right)$
             &15.750000\\
     &&&\\
     6&    $\left(
     \begin{array}{ccc}
     0.6 & 0.2 & 0.2  \\
     \end{array}
     \right)$         &   $\left(
                                                         \begin{array}{ccc}
                                                          0.4 & 0.4 & 0.1  \\
                                                          0.0 & 0.8 & 0.1 \\
                                                          0.2 & 0.3 & 0.6  \\
                                                        \end{array}
                                                       \right)$
            &18.000000\\
     &&&\\
     7&    $\left(
     \begin{array}{ccc}
     0.5 & 0.3 & 0.2  \\
     \end{array}
     \right)$        &    $\left(
                                                          \begin{array}{ccc}
                                                          0.7 & 0.1 & 0.1  \\
                                                          0.1 & 0.6 & 0.1 \\
                                                          0.1 & 0.1 & 0.7  \\
                                                        \end{array}
                                                        \right)$
           &8.000000\\
     &&&\\
     8&    $\left(
     \begin{array}{ccc}
     0.4 & 0.2 & 0.4  \\
     \end{array}
     \right)$ & $\left(
                                                 \begin{array}{ccc}
                                                          0.5 & 0.3 & 0.1  \\
                                                          0.2 & 0.5 & 0.1 \\
                                                          0.2 & 0.4 & 0.3  \\
                                                        \end{array}
                                               \right)$
      & 6.988764\\
      &&&\\
     9&    $\left(
     \begin{array}{ccc}
     0.25 & 0.5 & 0.5  \\
     \end{array}
     \right)$ & $\left(
                                                 \begin{array}{ccc}
                                                          0.1 & 0.8 & 0.1  \\
                                                          0.0 & 0.7 & 0.1 \\
                                                          0.1 & 0.1 & 0.7  \\
                                                        \end{array}
                                               \right)$
      & 7.131147\\
      &&&\\
     10&   $\left(
     \begin{array}{ccc}
     0.3 & 0.3 & 0.4  \\
     \end{array}
     \right)$  &                                  $\left(
                                                 \begin{array}{ccc}
                                                          0.6 & 0.2 & 0.1  \\
                                                          0.4 & 0.3 & 0.1 \\
                                                          0.2 & 0.0 & 0.8  \\
                                                        \end{array}
                                               \right)$
      & 13.863636\\\hline

\end{tabular}
\end{tiny}
\end{table}

\newpage

\begin{table}[h!]
\centering
\begin{tiny}

\caption{Probability distribution at service completion epoch for single vacation}\lb{tb2}$\vspace{0.03cm}$
\begin{tabular}{|c|c|c|c|c|c|c|c|c|} \hline
$n$& $p_{n,5}^+$& $p_{n,6}^+$& $p_{n,7}^+$& $p_{n,8}^+$& $p_{n,9}^+$& $p_{n,10}^+$& $p_n^+$&$Q_n^+$\\\hline
0&0.014648&0.002924&0.004521&0.003994&0.003179&0.002032&0.031298&0.007223\\
1&0.026816&0.005430&0.008134&0.007376&0.006159&0.005402&0.059317&0.024801\\
2&0.022733&0.004716&0.006452&0.005946&0.005253&0.007744&0.052844&0.039238\\
3&0.019845&0.004181&0.005053&0.004472&0.003974&0.009505&0.047030&0.044177\\
4&0.017456&0.003731&0.003935&0.003325&0.002963&0.010823&0.042233&0.044387\\
5&0.015384&0.003338&0.003056&0.002467&0.002200&0.011786&0.038231&0.033647\\
25&0.001244&0.000368&0.000019&0.000006&0.000005&0.007315&0.008957&0.000000\\
50&0.000054&0.000023&0.000000&0.000000&0.000000&0.001431&0.001508&0.000000\\
75&0.000002&0.000001&0.000000&0.000000&0.000000&0.000234&0.000237&0.000000\\
100&0.000000&0.000000&0.000000&0.000000&0.000000&0.000036&0.000036&0.000000\\
125&0.000000&0.000000&0.000000&0.000000&0.000000&0.000005&0.000005&0.000000\\
$\geq$150&0.000000&0.000000&0.000000&0.000000&0.000000&0.000000&0.000000&0.000000\\\hline
\end{tabular}

\caption{Probability distribution at arbitrary slot for single vacation }\lb{tb3}$\vspace{0.03cm}$
\begin{tabular}{|c|c|c|c|c|c|c|c|c|c|c|} \hline
$n$& $p_{n,0}$ & $p_{n,5}$ & $p_{n,6}$& $p_{n,7}$& $p_{n,8}$& $p_{n,9}$&$p_{n,10}$&$p_n^{queue}$&$Q_n$ \\\hline
0&0.001188&0.035713&0.008233&0.006116&0.005046&0.004414&0.004062&0.068733&0.003961\\
1&0.005268&0.031300&0.007339&0.004777&0.003832&0.003400&0.007160&0.072716&0.009640\\
2&0.011724&0.027560&0.006563&0.003716&0.002854&0.002536&0.009545&0.076377&0.011879\\
3&0.018993&0.024295&0.005875&0.002884&0.002117&0.001882&0.011362&0.079756&0.012348\\
4&0.026296&0.021423&0.005261&0.002237&0.001571&0.001395&0.012721&0.082899&0.011995\\
5&&0.018892&0.004712&0.001734&0.001164&0.001033&0.013704&0.047697&0.006458\\
25&&0.001528&0.000519&0.000010&0.000003&0.000002&0.008193&0.010255&0.000000\\
50&&0.000066&0.000033&0.000000&0.000000&0.000000&0.001596&0.001695&0.000000\\
75&&0.000002&0.000002&0.000000&0.000000&0.000000&0.000260&0.000264&0.000000\\
100&&0.000000&0.000000&0.000000&0.000000&0.000000&0.000040&0.000040&0.000000\\
$\geq$150&&0.000000&0.000000&0.000000&0.000000&0.000000&0.000000&0.000000&0.000000\\\hline
\multicolumn{10}{|c|}{$L$=18.439001,~~ $L_q$=11.796469,~~ $L_{s}$=7.943251,}\\
\multicolumn{10}{|c|}{$P_{busy}$=0.872710,~~~ $W$=36.878002,~~ $W_q$=23.592938}\\\hline
\end{tabular}
\end{tiny}
\end{table}

\begin{table}[h!]
\centering
\begin{tiny}
\caption{ Probability distribution at service completion epoch for multiple vacation}\lb{tb4}$\vspace{0.03cm}$

\begin{tabular}{|c|c|c|c|c|c|c|c|c|} \hline
$n$& $p_{n,5}^+$& $p_{n,6}^+$& $p_{n,7}^+$& $p_{n,8}^+$& $p_{n,9}^+$& $p_{n,10}^+$& $p_n^+$&$Q_n^+$\\\hline
0&0.006982&0.003847&0.005085&0.003905&0.002780&0.001638&0.024237&0.007271\\
1&0.012783&0.007145&0.009149&0.007212&0.005385&0.004275&0.045949&0.028328\\
2&0.010836&0.006206&0.007257&0.005814&0.004593&0.005985&0.040691&0.054321\\
3&0.009459&0.005501&0.005684&0.004373&0.003474&0.007211&0.035702&0.077241\\
4&0.008320&0.004911&0.004426&0.003251&0.002590&0.008093&0.031591&0.097430\\
5&0.007333&0.004393&0.003438&0.002412&0.001923&0.008712&0.028211&0.082236\\
25&0.000593&0.000484&0.000021&0.000006&0.000005&0.005163&0.006272&0.000000\\
50&0.000025&0.000030&0.000000&0.000000&0.000000&0.001019&0.001075&0.000000\\
75&0.000001&0.000002&0.000000&0.000000&0.000000&0.000168&0.000171&0.000000\\
100&0.000000&0.000000&0.000000&0.000000&0.000000&0.000026&0.000026&0.000000\\
125&0.000000&0.000000&0.000000&0.000000&0.000000&0.000004&0.000004&0.000000\\
$\geq$150&0.000000&0.000000&0.000000&0.000000&0.000000&0.000000&0.000000&0.000000\\\hline
\end{tabular}

\caption{Probability distribution at arbitrary slot for multiple vacation }\lb{tb5}$\vspace{0.03cm}$

\begin{tabular}{|c|c|c|c|c|c|c|c|c|c|} \hline
$n$&  $p_{n,5}$ & $p_{n,6}$& $p_{n,7}$& $p_{n,8}$& $p_{n,9}$&$p_{n,10}$&$p_n^{queue}$&$Q_n$ \\\hline
0&0.022959&0.014610&0.009278&0.006653&0.005205&0.004416&0.068500&0.005379\\
1&0.020123&0.013024&0.007248&0.005053&0.004010&0.007570&0.072603&0.015575\\
2&0.017718&0.011647&0.005637&0.003763&0.002991&0.009875&0.076236&0.024605\\
3&0.015619&0.010427&0.004376&0.002792&0.002220&0.011561&0.079522&0.032527\\
4&0.013772&0.009337&0.003393&0.002071&0.001645&0.012774&0.082530&0.039538\\
5&0.012145&0.008362&0.002631&0.001536&0.001218&0.013619&0.060802&0.021289\\
25&0.000983&0.000923&0.000016&0.000004&0.000003&0.007798&0.009727&0.000000\\
50&0.000042&0.000058&0.000000&0.000000&0.000000&0.001533&0.001633&0.000000\\
75&0.000004&0.000000&0.000000&0.000000&0.000000&0.000252&0.000256&0.000000\\
100&0.000000&0.000000&0.000000&0.000000&0.000000&0.000039&0.000039&0.000000\\
$\geq$150&&0.000000&0.000000&0.000000&0.000000&0.000000&0.000000&0.000000\\\hline
\multicolumn{9}{|c|}{$L$=18.915229,~~ $L_q$=12.142686,~~ $L_{s}$=7.760356,}\\
\multicolumn{9}{|c|}{$P_{busy}$=0.836248,~~~$W$=37.830459,~~ $W_q$=24.285373 }\\\hline
\end{tabular}
\end{tiny}
\end{table}

\newpage
\textbf{Example 2:} In this example, the service time has been considered as negative binomial (NB) distribution with pmf
\bea &&s_i(n)=\binom{n-1}{r-1}\mu_i^r(1-\mu_i)^{n-r},~n=r,r+1,r+2,\ldots ~ \mbox{which consequently leads}\nn\\&&~ K^{(i)}(x)=\left\{\frac{\mu_i(1-\la+\la x)}{1-(1-\mu_i)(1-\la+\la x)}\right\}^r.\nn\eea
In particular, here $r=3,~\la=0.15,~a=12,~b=20$ and service rates are taken as follows: $\mu_{12}=0.346153$, $\mu_{13}=0.321428$, $\mu_{14}=0.300000$, $\mu_{15}=0.281250$, $\mu_{16}=0.264705$, $\mu_{17}=0.250000$, $\mu_{18}=0.236842$, $\mu_{19}=0.225000$, $\mu_{20}=0.214285$. On the other hand, the
vacation time follows PH$_D$ distribution with $\boldsymbol{\beta}=(0.2,0.2,0.6)$, $\textbf{T}=\left(
                                                                           \begin{array}{ccc}
                                                                             0.5 & 0.3 & 0.2 \\
                                                                              0.4& 0.3 & 0.1 \\
                                                                              0.3 & 0.2 & 0.1 \\
                                                                           \end{array}
                                                                         \right)
$, mean=17.736842 so that the traffic intensity of the system $\rho=0.315000$. The joint distribution at service completion and arbitrary epochs along with some performance measures are
displayed in Tables [\ref {tb6}-\ref {tb9}].
\begin{table}[h!]
\centering
\begin{tiny}
    \caption{Probability distribution at service completion epoch for single vacation}\lb{tb6}$\vspace{0.03cm}$

\begin{tabular}{|c|c|c|c|c|c|c|c|c|c|c|c|} \hline
$n$& $p_{n,12}^+$& $p_{n,13}^+$& $p_{n,14}^+$& $p_{n,15}^+$& $p_{n,16}^+$& $p_{n,17}^+$&$p_{n,18}^+$&$p_{n,19}^+$&$p_{n,20}^+$& $p_n^+$&$Q_n^+$\\\hline
0&0.142104&0.000858&0.000564&0.000372&0.000245&0.000162&0.000107&0.000071&0.000047&0.144530&0.033265\\
1&0.169353&0.001073&0.000737&0.000506&0.000346&0.000237&0.000162&0.000110&0.000109&0.172633&0.071484\\
2&0.104665&0.000706&0.000513&0.000369&0.000265&0.000189&0.000134&0.000094&0.000144&0.107079&0.085673\\
3&0.046869&0.000340&0.000263&0.000200&0.000151&0.000113&0.000083&0.000061&0.000147&0.048227&0.078642\\
4&0.017560&0.000137&0.000114&0.000092&0.000074&0.000057&0.000044&0.000034&0.000130&0.018242&0.063170\\
5&0.005904&0.000050&0.000044&0.000038&0.000032&0.000026&0.000021&0.000017&0.000106&0.006238&0.047482\\
8&0.000157&0.000002&0.000002&0.000002&0.000002&0.000002&0.000002&0.000002&0.000045&0.000216&0.017674\\
12&0.000000&0.000000&0.000000&0.000000&0.000000&0.000000&0.000000&0.000000&0.000011&0.000011&0.004487\\
16&0.000000&0.000000&0.000000&0.000000&0.000000&0.000000&0.000000&0.000000&0.000003&0.000003&0.001132\\
30&0.000000&0.000000&0.000000&0.000000&0.000000&0.000000&0.000000&0.000000&0.000000&0.000000&0.000009\\
$\geq$40&0.000000&0.000000&0.000000&0.000000&0.000000&0.000000&0.000000&0.000000&0.000000&0.000000&0.000000\\\hline
\end{tabular}
$\vspace{0.03cm}$
    \caption{Probability distribution at arbitrary slot for single vacation }\lb{tb7}$\vspace{0.01cm}$
\begin{tabular}{|c|c|c|c|c|c|c|c|c|c|c|c|c|c|} \hline
$n$& $p_{n,0}$ & $p_{n,12}$ & $p_{n,13}$& $p_{n,14}$& $p_{n,15}$& $p_{n,16}$&$p_{n,17}$&$p_{n,18}$&$p_{n,19}$&$p_{n,20}$&$Q_n$&$p_n^{queue}$ \\\hline
0&0.005511&0.057487&0.000386&0.000281&0.000203&0.000147&0.000106&0.000076&0.000055&0.000039&0.018435&0.082726\\
1&0.017355&0.029427&0.000208&0.000158&0.000120&0.000090&0.000067&0.000049&0.000036&0.000055&0.035194&0.082759\\
2&0.031550&0.012086&0.000091&0.000073&0.000058&0.000046&0.000035&0.000027&0.000021&0.000055&0.038741&0.082783\\
3&0.044580&0.004320&0.000035&0.000030&0.000025&0.000021&0.000017&0.000013&0.000011&0.000048&0.033702&0.082802\\
4&0.055047&0.001410&0.000012&0.000011&0.000010&0.000009&0.000008&0.000006&0.000005&0.000038&0.026258&0.082814\\
5&0.062914&0.000432&0.000004&0.000004&0.000003&0.000003&0.000003&0.000002&0.000002&0.000029&0.019425&0.082826\\
6&0.068641&0.000126&0.000001&0.000001&0.000001&0.000001&0.000001&0.000001&0.000001&0.000001&0.014032&0.082807\\
7&0.072750&0.000035&0.000000&0.000000&0.000000&0.000000&0.000000&0.000000&0.000000&0.000015&0.010030&0.082830\\
8&0.075679&0.000009&0.000000&0.000000&0.000000&0.000000&0.000000&0.000000&0.000000&0.000011&0.007138&0.082837\\
9&0.077760&0.000003&0.000000&0.000000&0.000000&0.000000&0.000000&0.000000&0.000000&0.000008&0.005070&0.082841\\
10&0.079238&0.000000&0.000000&0.000000&0.000000&0.000000&0.000000&0.000000&0.000000&0.000005&0.003598&0.082841\\
11&0.080286&0.000000&0.000000&0.000000&0.000000&0.000000&0.000000&0.000000&0.000000&0.000004&0.002553&0.082843\\
12&&0.000000&0.000000&0.000000&0.000000&0.000000&0.000000&0.000000&0.000000&0.000003&0.001809&0.001812\\
20&&0.000000&0.000000&0.000000&0.000000&0.000000&0.000000&0.000000&0.000000&0.000000&0.000115&0.000115\\
30&&0.000000&0.000000&0.000000&0.000000&0.000000&0.000000&0.000000&0.000000&0.000000&0.000004&0.000004\\
$\geq$40&&0.000000&0.000000&0.000000&0.000000&0.000000&0.000000&0.000000&0.000000&0.000000&0.000000&0.000000\\\hline
\multicolumn{13}{|c|}{$L$=6.866692,~~ $L_q$=5.556901,~~ $L_{s}$=12.095813,}\\
\multicolumn{13}{|c|}{$P_{busy}$=0.108284,~~~ $W$=45.777950,~~ $W_q$=37.046006}\\\hline
\end{tabular}
\end{tiny}
\end{table}

\begin{table}[h!]
\centering
\begin{tiny}
    \caption{ Probability distribution at service completion epoch for multiple vacation}\lb{tb8}$\vspace{0.03cm}$

\begin{tabular}{|c|c|c|c|c|c|c|c|c|c|c|c|} \hline
$n$& $p_{n,12}^+$& $p_{n,13}^+$& $p_{n,14}^+$& $p_{n,15}^+$& $p_{n,16}^+$& $p_{n,17}^+$&$p_{n,18}^+$&$p_{n,19}^+$&$p_{n,20}^+$& $p_n^+$&$Q_n^+$\\\hline
0&0.014508&0.009600&0.006318&0.004163&0.002747&0.001816&0.001202&0.000797&0.000529&0.041680&0.012462\\
1&0.017291&0.012009&0.008260&0.005665&0.003878&0.002652&0.001812&0.001237&0.001219&0.054023&0.031604\\
2&0.010686&0.007896&0.005740&0.004139&0.002966&0.002114&0.001500&0.001061&0.001612&0.037714&0.046955\\
3&0.004785&0.003799&0.002944&0.002247&0.001695&0.001266&0.000938&0.000689&0.001646&0.020009&0.055918\\
4&0.001792&0.001538&0.001276&0.001035&0.000825&0.000647&0.000501&0.000384&0.001458&0.009456&0.060383\\
5&0.000602&0.000560&0.000499&0.000431&0.000363&0.000300&0.000243&0.000194&0.001187&0.004379&0.062482\\
10&0.000001&0.000002&0.000003&0.000003&0.000003&0.000003&0.000003&0.000003&0.000260&0.000281&0.064591\\
15&0.000000&0.000000&0.000000&0.000000&0.000000&0.000000&0.000000&0.000000&0.000047&0.000047&0.017898\\
20&0.000000&0.000000&0.000000&0.000000&0.000000&0.000000&0.000000&0.000000&0.000008&0.000008&0.003202\\
30&0.000000&0.000000&0.000000&0.000000&0.000000&0.000000&0.000000&0.000000&0.000000&0.000000&0.000102\\
$\geq$45&0.000000&0.000000&0.000000&0.000000&0.000000&0.000000&0.000000&0.000000&0.000000&0.000000&0.000000\\\hline
\end{tabular}
\end{tiny}
\end{table}
\newpage
\begin{table}
\centering
\begin{tiny}
    \caption{Probability distribution at arbitrary slot for multiple vacation }\lb{tb9}$\vspace{0.03cm}$
\begin{tabular}{|c|c|c|c|c|c|c|c|c|c|c|c|c|} \hline
$n$ & $p_{n,12}$ & $p_{n,13}$& $p_{n,14}$& $p_{n,15}$& $p_{n,16}$&$p_{n,17}$&$p_{n,18}$&$p_{n,19}$&$p_{n,20}$&$Q_n$& $p_n^{queue}$\\\hline
0&0.014370&0.010580&0.007706&0.005590&0.004045&0.002920&0.002104&0.001514&0.001087&0.016909&0.066825\\
1&0.007356&0.005708&0.004355&0.003292&0.002471&0.001844&0.001369&0.001012&0.001516&0.038825&0.067748\\
2&0.003021&0.002505&0.002026&0.001613&0.001268&0.000986&0.000760&0.000581&0.001516&0.054125&0.068401\\
3&0.001080&0.000964&0.000832&0.000701&0.000580&0.000473&0.000380&0.000302&0.001312&0.062244&0.068868\\
4&0.000352&0.000340&0.000314&0.000281&0.000246&0.000210&0.000176&0.000146&0.001049&0.066081&0.069195\\
5&0.000108&0.000113&0.000112&0.000106&0.000098&0.000088&0.000078&0.000067&0.000800&0.067859&0.069429\\
10&0.000000&0.000000&0.000000&0.000000&0.000000&0.000000&0.000000&0.000000&0.000159&0.069735&0.069894\\
15&0.000000&0.000000&0.000000&0.000000&0.000000&0.000000&0.000000&0.000000&0.000028&0.017674&0.017702\\
20&0.000000&0.000000&0.000000&0.000000&0.000000&0.000000&0.000000&0.000000&0.000005&0.003162&0.003167\\
30&0.000000&0.000000&0.000000&0.000000&0.000000&0.000000&0.000000&0.000000&0.000000&0.000101&0.000101\\
$\geq$45&0.000000&0.000000&0.000000&0.000000&0.000000&0.000000&0.000000&0.000000&0.000000&0.000000&0.000000\\\hline
\multicolumn{12}{|c|}{$L$=8.628551,~~ $L_q$=7.060092,~~ $L_{s}$=14.658480,}\\
\multicolumn{12}{|c|}{$P_{busy}$=0.107000,~~~ $W$=57.523677,~~ $W_q$=47.067284}\\\hline
\end{tabular}
\end{tiny}
\end{table}
After the tabular representation of the numerical results, we now turn our focus to study the behavior of the system through some graphical representation. In order to show the significance of our proposed model, we figure out a comparison between batch-size dependent and independent service policies by considering two cases, viz. \\
\emph{Case 1}: The service times of the batches vary with the size of the batch undergoing service.\\
\emph{Case 2}: The mean service times of the batches are the same irrespective of the size of the batches. The same service rates for all the batches are the weighted average of the service rates of Case 1. The weighted average is calculated by the formula $\frac{\sum_{r=a}^{b}r\mu_r}{\sum_{r=a}^{b}r}$. \\
\hspace*{0.3cm}For figures \ref{fig1} - \ref{fig3}, the input parameters are taken as: service time follows NB distribution, vacation time follows geometric distribution, $a=8,~b=15$, and service rates for two cases are precisely given in Table \ref{tb10}.
\begin{table}[h!]
\centering
\caption{Service rates for Case 1 and Case 2}\lb{tb10}
\begin{tiny}
\begin{tabular}{|c|c|c|}\hline
   $r$ & $\mu_r$ (Case 1) & $\mu$ (Case 2) \\\hline
   8& 0.277778 & 0.199378 \\
   9& 0.250000 & 0.199378 \\
   10& 0.227273& 0.199378 \\
   11& 0.208333 & 0.199378 \\
   12& 0.192307 & 0.199378 \\
   13& 0.178571 &  0.199378\\
   14& 0.166666 &  0.199378\\
   15& 0.156250 & 0.199378 \\
  \hline
\end{tabular}
\end{tiny}
\end{table}
In figures \ref{fig1} and \ref{fig2}, the average number of customers in the queue at arbitrary slot ($L_q$) and the probability that the server is busy ($P_{busy}$) are depicted versus the traffic load $\rho$, respectively. From the figure \ref{fig1} and \ref{fig2}, it is easily visible that $L_q$ and $P_{busy}$ are higher for the Case 2 as compared to Case 1 which suggests that batch-size-dependent service time is more significant rather than batch-size-independent service policy. Also it may be noted here that $L_q$ is slightly greater in case of multiple vacation which is in the expected line as the availability of the server in case of single vacation is more compared to multiple vacation. Similar type of conclusion has been drawn in Gupta et al. \cite{gupta2019finite}.\\
\hspace*{0.3cm}From the figure \ref{fig2}, one may observe that $P_{busy}$ is higher for single vacation than the corresponding multiple vacation for Case 2 and both of them become the equal and stable as $\rho$ increases. This behavior of the system is also in the expected line and Gupta et al. \cite{gupta2019finite} and Samanta et al. \cite{samanta2007discrete} pointed out the similar type of observation.\\
\hspace*{0.3cm}In figure \ref{fig3}, we sketch the average waiting time of a customer in the system for different values of $\la$. One can notice that waiting time is less in Case 1 compared to Case 2, which eventually justifies the potentiality of batch-size-dependent service policy from the application perspective. Moreover, it may be pointed out that a customer has to wait more time in the system in case of multiple vacation which is quite obvious behavior of the system.

\begin{figure}[h!]
\begin{center}
\includegraphics[scale=0.4]{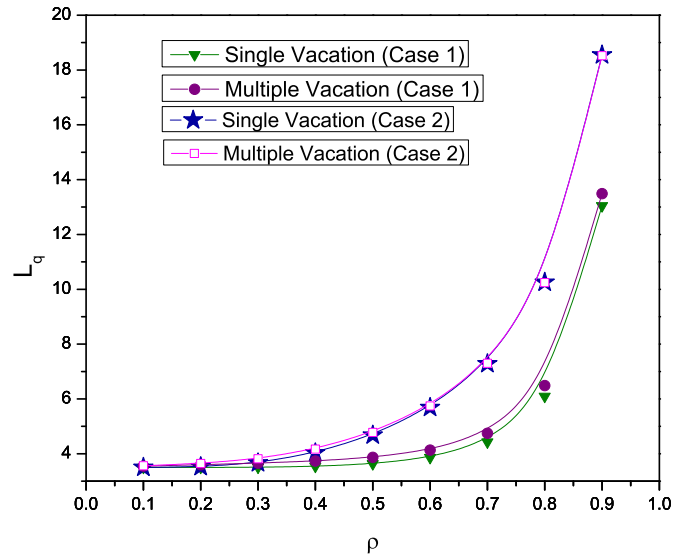}
\end{center}
\caption{Number of customers in the queue versus traffic intensity }\lb{fig1}
\end{figure}

\begin{figure}[h!]
\begin{center}
\includegraphics[scale=0.4]{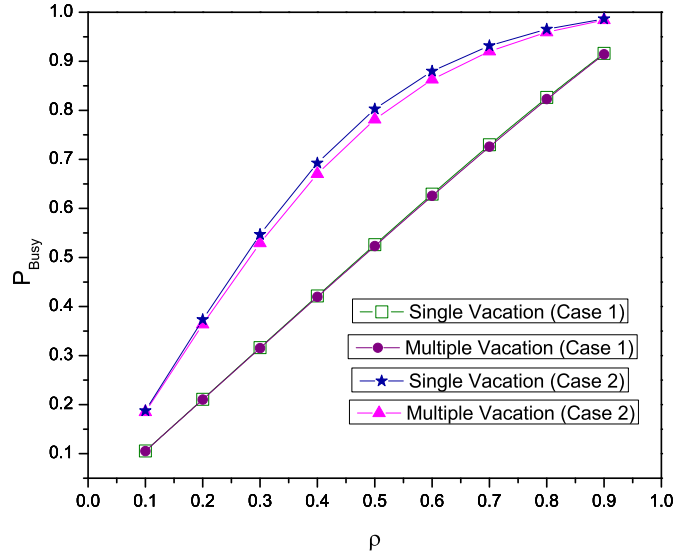}
\end{center}
\caption{Busy probability of the server versus traffic intensity }\lb{fig2}
\end{figure}

\begin{figure}[h!]
\begin{center}
\includegraphics[scale=0.4]{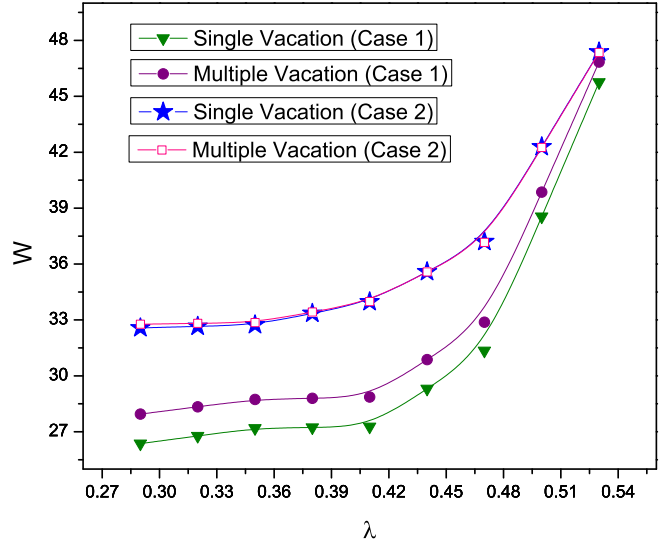}
\end{center}
\caption{ Waiting time of a customer in the system versus arrival rate}\lb{fig3}
\end{figure}
\newpage
\subsection{System cost optimization}
It is frequently observed that the system designers or vendors want to minimize the total system cost at the pre-implantation stage. Keeping this in mind, in this section, our foremost concern is to demonstrate a comparison between the batch-size-dependent (the present model, (Case 1)) and batch-size-independent model (Case 2) through the cost of the corresponding models. We first associate some cost to the system characteristics viz.,

\begin{itemize}

\item $C_w$ be the waiting/holding cost per customer per unit time in the queue.

\item $C_{idle}$ be the idleness cost of the server per unit time.

\item $C_{operating}$ be the operating (serving/transmission) cost per unit time.

\end{itemize}
Hence in the long run the total system cost ($TSC$) is presented by:
\bea TSC=C_w~.~ L_q + C_{idle}~.~\biggl((1-\delta_{p})P_{dor}+Q_{vac}\biggr) + C_{operating}~.~L_s \lb{mmor1}\eea
Determination of the optimal value of the cost function through some conventional optimization techniques is very difficult as the analytic expressions of $L_q$, $P_{dor},~Q_{vac}$, $L_s$ are not in our grip. On account of this, from the expression presented in (\ref{mmor1}), the associated cost of the system is calculated numerically. We consider the cost parameter as : $C_w=5$ unit, $C_{idle}=2$ unit, $C_{operating}=1$ unit, and other parameters are the same as taken for figures \ref{fig1} - \ref{fig3}. \\

\begin{figure}[h!]
\begin{center}
\includegraphics[scale=0.4]{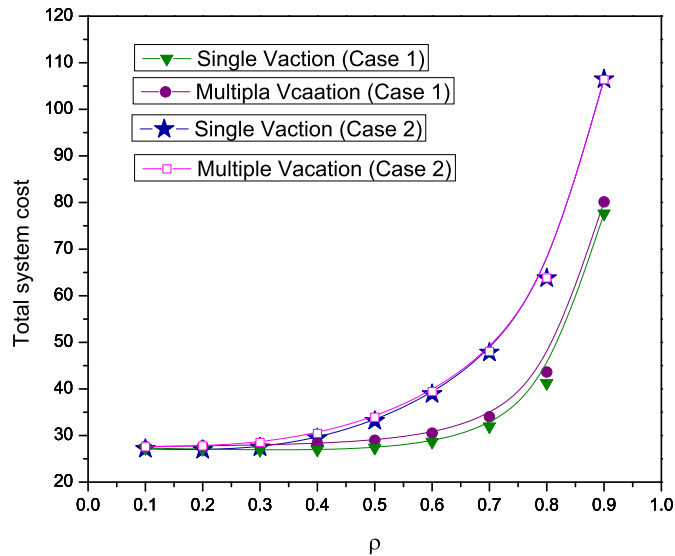}
\end{center}
\caption{Total system cost versus traffic intensity }\lb{fig4}
\end{figure}

The figure \ref{fig4} exhibits that the total system cost is always less in Case 1 compared to Case 2. This outcome suggests that vendor/system designer should adopt the batch-size-dependent service policy as it produces less cost rather than batch-size-independent service policy. It may also be remarked here that total system cost is higher for the multiple vacation compared to single vacation. This behavior is also relevant in the sense that the more availability of the server in case of single vacation produces less system cost.

\section{Conclusion}\lb{CN}
In this paper, we have analyzed a batch service queue with infinite waiting space, Bernoulli arrival process, batch-size-dependent service policy, single and multiple vacation. The generating function approach with double variables is adopted to derive the bivariate probability generating function of queue length and server content distribution at service completion epoch. We have also presented the complete extraction of the probabilities in a simple and elegant way. The
joint distribution at arbitrary slot is also acquired using which several noteworthy marginal distribution and performance measures are reported. The analytic procedure and results are illustrated numerically and graphically through the inclusion of significant distributions which can cope with real-life circumstances. Towards the end it is hoped that the results will be fruitful to the vendors/system designers. The inclusion of the batch Bernoulli arrival process or discrete Markovian arrival process will make the investigation very interesting, and in future, we will keep track for the analysis of those models with single/multiple vacation.
~\\
\bibliographystyle{unsrtnat}
\bibliography{discretevacation}

\end{document}